\documentclass[12pt,reqno]{amsart}
\usepackage{latexsym}
\usepackage{amssymb}
\usepackage{amscd}

\numberwithin{equation}{section}

\newtheorem{theorem}{Theorem}
\newtheorem{proposition}[theorem]{Proposition}
\newtheorem{lemma}[theorem]{Lemma}
\newtheorem{corollary}[theorem]{Corollary}
\newtheorem{conjecture}[theorem]{Conjecture}

\theoremstyle{remark}
\newtheorem*{remark}{Remark}
\newtheorem{remarknu}[theorem]{Remark}

\def\leaderfill{\leaders\hbox to 3mm{\hss.\hss}\hfill}
\def\po#1#2{(#1)_#2}

\def\fl#1{\left\lfloor#1\right\rfloor}

\def\coef#1{\left\langle#1\right\rangle}

\def\al{\alpha}
\def\be{\beta}
\def\de{\delta}

\def\ga{\gamma}

\newcommand{\Z}{\mathbb{Z}}

\begin{document}
\title[Recursive sequences modulo prime powers]{A 
method for determining the mod-$p^k$ behaviour of
recursive sequences}
\author[C. Krattenthaler and 
T.\,W. M\"uller]{C. Krattenthaler$^{\dagger}$ and
T. W. M\"uller$^*$} 

\address{$^{\dagger*}$Fakult\"at f\"ur Mathematik, Universit\"at Wien,
Oskar-Morgenstern-Platz~1, A-1090 Vienna, Austria.
WWW: {\tt http://www.mat.univie.ac.at/\lower0.5ex\hbox{\~{}}kratt}.}

\thanks{$^\dagger$Research partially supported by the Austrian
Science Foundation FWF, grants Z130-N13 and F50-N15,
the latter in the framework of the Special Research Program
``Algorithmic and Enumerative Combinatorics"}
\thanks{$^\ast$Research supported by Lise Meitner Grant M1661-N25 of
  the Austrian Science Foundation FWF} 



\dedicatory{Dedicated to the memory of J\'anos Acz\'el at
the occasion of the 100th anniversary of his birth}

\begin{abstract}
We present a method for obtaining congruences modulo powers of a
prime number~$p$ for combinatorial sequences whose generating
function satisfies an algebraic differential equation.
This method generalises the one by Kauers and the authors 
[{\it Electron.\ J. Combin.}\ {\bf18}(2) (2012), Art.~P37]
from $p=2$ to arbitrary primes. Our applications include
congruences for numbers of non-crossing graphs and numbers
of Kreweras walks modulo powers of~$3$, as well as congruences
for Fu\ss--Catalan numbers and blossom tree numbers modulo
powers of arbitrary primes.
\end{abstract}
\maketitle

\section{Introduction}
The purpose of the present article is to generalise the basic theory
concerning congruences for combinatorial sequences developed in
\cite{KaKMAA} for $2$-power moduli to arbitrary prime power moduli,
and to discuss a number of further applications of our method.  

Given a prime number $p$, consider the (formal) power series
\begin{equation} \label{eq:Phidef} 
\Phi(z) = \Phi_p(z) := \sum_{n\geq0} z^{p^n}.
\end{equation}
The series $\Phi_p(z)$ is easily seen to be transcendental over
$\mathbb{Z}$; this can be shown, for instance, by generalising (in a
straightforward manner) the density argument given in
\cite[Lemma~1]{KaKMAA}. In Lemma~\ref{lem:minpol} below, we provide a
different, and more elegant, argument, which however presupposes a
number of non-trivial facts concerning certain auxiliary series
$H_{b_1, b_2,\ldots, b_r}(z)$ established in Section~\ref{sec:extr}.  

While $\Phi_p(z)$ is transcendental over $\mathbb{Z}$, it is easily
seen to be algebraic when considered over finite rings
$\mathbb{Z}/p^\gamma\mathbb{Z}$, and the first focus of our paper is
on polynomial identities for $\Phi_p(z)$ modulo a given $p$-power, which
are of \textit{minimal degree}; cf.\ the next section. A (still
stubbornly unproven) conjecture concerns this minimal degree of
$\Phi_p(z)$ modulo~$p^\gamma$, see Conjecture~\ref{conj:1} in the
next section. As in \cite[Sec.~2]{KaKMAA} for the case of the
prime~$p=2$, we are only able to establish that the conjectured
minimal degree constitutes a lower bound, see Lemma~\ref{lem:minpol}.

The basic idea of our method for producing and establishing
congruences modulo $p$-powers for combinatorial sequences is to
express solutions of appropriate (formal) differential equations
(in technical terms: differentially algebraic power series over
the integers satisfying a uniqueness condition) 
as \textit{polynomials in $\Phi_p(z)$},
\textit{with coefficients that are Laurent polynomials in $z$}. For
each such differential equation, our method may be implemented to
(semi-)automatically yield such polynomial expressions in $\Phi_p(z)$
for arbitrary $p$-power moduli, provided such expressions do in fact
exist. 

However, such a procedure as described in the last paragraph would be
without any practical value, if we could not, at the same time,
provide an efficient algorithm for \textit{extracting coefficients
  from powers of\/ $\Phi_p(z)$}. This is the topic of
Section~\ref{sec:extr}. Briefly, the solution to this problem we
propose is to (i) \textit{expand powers of $\Phi_p(z)$} into
$\mathbb{Z}$-linear sums of certain auxiliary series  
\begin{equation}
\label{Eq:HDef}
H_{b_1, b_2,\ldots, b_r}(z):= \sum_{n_1>n_2>\cdots > n_r\geq0}
z^{b_1p^{n_1} + b_2 p^{n_2} + \cdots + b_r p^{n_r}};
\end{equation}
(ii) to show that the set consisting of $1$ and the series
(\ref{Eq:HDef}) with parameters $b_1, b_2, \ldots, b_r$ relatively
prime to $p$ are \textit{linearly independent over the ring
  $\mathbb{Z}[z, z^{-1}]$, as well as over the rings
  $(\mathbb{Z}/p^\gamma\mathbb{Z})[z,z^{-1}]$}; (iii) to prove that a series
$H_{b_1, b_2, \ldots, b_r}(z)$, without restrictions on its parameters
$b_1, b_2, \ldots, b_r$, can always \textit{be expressed as a linear
  combination over $\mathbb{Z}[z, z^{-1}]$} of these linearly
independent series; 
and (iv) to demonstrate that \textit{coefficient
  extraction from a series {\em (\ref{Eq:HDef})} with all parameters
  $b_1, b_2, \ldots, b_r$ relatively prime to $p$ can be done
  effectively}; cf.\ Equation~(\ref{eq:Phipot}),
Corollary~\ref{lem:Hind}, Proposition~\ref{conj:H}, and
Remark~\ref{rem:eff}. 

A detailed description of our method is given in
Section~\ref{sec:method}, while the remaining sections treat a number
of applications beyond the case where $p=2$: in
Section~\ref{sec:noncr}, we obtain congruences modulo arbitrary
$3$-powers for the number of non-crossing graphs with a given number
of vertices; Section~\ref{sec:Krew} deals with congruences modulo
$3$-powers for the number of Kreweras walks of given length in the
plane; Section~\ref{sec:FCat} treats the Fu{\ss}--Catalan numbers 
\[
F(n;k) := \frac{1}{n} \binom{kn}{n-1}
\] 
modulo powers of a prime number $p$, provided that the parameter~$k$
is itself a power of~$p$; finally, the last section studies the number 
\[
B(n;p) := \frac{p+1}{n((p-1)n +2)} \binom{pn}{n-1}
\]
of $p$-ary blossom trees with $n$ white nodes modulo arbitrary powers
of $p$. See, in particular, Theorems~\ref{thm:N}, \ref{thm:noncr27},
\ref{thm:K}, \ref{thm:K27}, \ref{thm:FCat}, \ref{thm:FCat27},
\ref{thm:Schaeff}, and \ref{thm:S27}.  

We would like to point out that Rowland and Yassawi \cite{RoYaAA} have
developed a completely different method for proving congruences.
Their method applies to diagonals of rational functions, and 
is based on the construction of automata. Our method is complementary:
there is a considerable intersection of applications (in particular,
if one also adds the variation \cite{KrMuAE} of the method presented
here), but neither can our method cover all diagonals of rational
functions nor can their method cover all differentially algebraic
series. Applications of our method not covered by \cite{RoYaAA}
can be found in \cite[Sections~8--14]{KaKMAA}.

\section{The $p$-power series $\Phi_p(z)$}
\label{sec:Phi}

The series $\Phi_p(z)$ in \eqref{eq:Phidef} 
is the principal character in the method for determining
congruences of recursive sequences modulo $p$-powers
which we describe in Section~\ref{sec:method}. 
For the sake of better readability, we shall usually suppress
the index~$p$ and simply write $\Phi(z)$ for $\Phi_p(z)$.
It is not hard to see that this series is transcendental over
$\mathbb Z[z]$ (this follows for instance from 
Lemma~\ref{lem:minpol} below). However, if the
coefficients of $\Phi(z)$ are considered modulo a $p$-power $p^\ga$,
then $\Phi(z)$ obeys a polynomial relation with coefficients that are
polynomials in $z$. The focus of this section is on what may be said 
concerning such polynomial relations, and, in particular, about those of
minimal length. In the proofs of Lemma~\ref{lem:minpol} and
Proposition~\ref{prop:minpol} below, we shall already 
make use of some fundamental results which will only be established
later in Section~\ref{sec:extr}. This will not create any circular
arguments since Lemma~\ref{lem:minpol} and
Proposition~\ref{prop:minpol} will not be needed in Section~\ref{sec:extr}.

\medskip
Here and in the sequel,
given power series (or Laurent series) $f(z)$ and $g(z)$,
we write 
$$f(z)=g(z)~\text {modulo}~p^\ga$$ 
to mean that the coefficients
of $z^i$ in $f(z)$ and $g(z)$ agree modulo~$p^\ga$ for all $i$.

We say that a polynomial $A(z,t)$ in 
$z$ and $t$ is {\it minimal for the modulus
$p^\ga$}, if it is monic (as a polynomial in $t$), 
has integral coefficients, satisfies
$A(z,\Phi(z))=0$~modulo~$p^\ga$, and there is no monic polynomial $B(z,t)$
with integral coefficients
of $t$-degree less than that of $A(z,t)$ with $B(z,\Phi(z))=0$~modulo~$p^\ga$.
(Minimal polynomials are not unique; see Remark~\ref{rem:min}.)
Furthermore, we let $v_p(\al)$ denote the $p$-adic valuation of the
integer $\al$, that is, the maximal exponent $e$ such that $p^e$
divides $\al$.

The lemma below provides a lower bound for the degree of a polynomial
that is minimal for the modulus $p^\ga$.

\begin{lemma} \label{lem:minpol}
If $A(z,t)$ is minimal for the modulus $p^\ga$, 
then the degree $d$ of $A(z,t)$ in $t$ satisfies $v_p(d!)\ge\ga$.
In particular, the series $\Phi(z)$ is transcendental over $\mathbb
Z[z]$.
\end{lemma}

\begin{proof}
Given positive integers $b_1,b_2,\dots,b_r$,
define the series $H_{b_1,b_2,\dots,b_r}(z)$ by
$$
H_{b_1,b_2,\dots,b_r}(z):=\sum_{n_1>n_2>\dots>n_r\ge0}
z^{b_1p^{n_1}+b_2p^{n_2}+\dots+b_rp^{n_r}}.
$$
By \eqref{eq:Phipot} and Proposition~\ref{conj:H}, a power
$\Phi^d(z)$ can be expressed as a linear combination of series
$H_{b_1,b_2,\dots,b_s}(z)$ and~$1$ with coefficients in $\Z[z]$. 
Furthermore, all $b_i$'s in these series are relatively prime to~$p$,
and $b_1+b_2+\dots+b_s\le d$.
In particular, we have
\begin{equation} \label{eq:Phid} 
\Phi^d(z)=d!\,H_{1,1,\dots,1}(z)+\text{other terms},
\end{equation}
with $d$ repetitions of $1$ in the index of $H$. Since, by
Corollary~\ref{lem:Hind}, 
the series $H_{b_1,b_2,\dots,b_s}(z)$ and~$1$ are linearly independent
over $(\Z/p^\ga\Z)[z]$, and
$H_{1,1,\dots,1}(z)$ (with $d$ repetitions of~$1$) does not appear
in the expansion of $\Phi^e(z)$ for $e<d$, a monic polynomial of
degree~$d$ in $\Phi(z)$ can only vanish modulo~$p^\ga$ if
$d!\equiv0$~(mod~$p^\ga$), or, equivalently, if $v_p(d!)\ge\ga$.
This establishes the assertion.
\end{proof}

\begin{proposition} \label{prop:minpol}
Minimal polynomials for the moduli
$p,p^2,p^3,\dots,p^{p+1}$ are
\begin{alignat*}2 
&t^p-t+z&&\text {\em modulo $p$},\\
&
(t^p-t+z)^2
&&\text {\em modulo $p^2$},\\
\multispan4{\hbox to 8cm{\leaderfill}}\\
&
(t^p-t+z)^p
&&\text {\em modulo $p^p$},\\
&
(t^p-t+z)^p-p^{p-1}(t^2-t+z)\quad 
&&\text {\em modulo $p^{p+1}$},\\
\end{alignat*}
\end{proposition}

\begin{proof}
By \eqref{eq:Phipot}, we have
\begin{align*}
\Phi^p(z)&=\sum_{r=1} ^p
\underset{b_1+\dots+b_r=p}{\sum _{b_1,\dots,b_r\ge1} ^{}}
\frac {p!}
{b_1!\,b_2!\cdots b_r!}H_{b_1,b_2,\dots,b_r}(z)\\
&=\sum_{n\ge0}z^{p^{n+1}}+p\sum_{r=2} ^p
\underset{b_1+\dots+b_r=p}{\sum _{b_1,\dots,b_r\ge1} ^{}}
\frac {(p-1)!}
{b_1!\,b_2!\cdots b_r!}H_{b_1,b_2,\dots,b_r}(z),
\end{align*}
or, equivalently,
\begin{equation} \label{eq:PhiGl} 
\Phi^p(z)-\Phi(z)+z=
p\sum_{r=2} ^p
\underset{b_1+\dots+b_r=p}{\sum _{b_1,\dots,b_r\ge1} ^{}}
\frac {(p-1)!}
{b_1!\,b_2!\cdots b_r!}H_{b_1,b_2,\dots,b_r}(z).
\end{equation}
Clearly, this implies the claim modulo~$p$, and, by potentiation, 
as well the claims modulo~$p^i$ for $i=2,3,\dots,p$.

Taking both sides of \eqref{eq:PhiGl} to the $p$-th power,
we obtain
$$
\left(\Phi^p(z)-\Phi(z)+z\right)^p=
p^p\sum_{r=2} ^p
\underset{b_1+\dots+b_r=p}{\sum _{b_1,\dots,b_r\ge1} ^{}}
\left(\frac {(p-1)!}
{b_1!\,b_2!\cdots b_r!}\right)^p
H^p_{b_1,b_2,\dots,b_r}(z)+O\left(p^{p+1}\right),
$$
where the expression $O\left(p^{p+1}\right)$ subsumes terms
all of whose coefficients are divisible by $p^{p+1}$.
If, in the sum over~$r$ on the right-hand side, we use
the congruences $N^p\equiv N$~(mod~$p$), which is valid for
any integer~$N$, and
$$
H^p_{b_1,b_2,\dots,b_r}(z)
=
H_{b_1,b_2,\dots,b_r}(z)+O\left(p\right),
$$
then, in combination with \eqref{eq:PhiGl} read from right to left,
we obtain the claim on the minimal polynomial modulo~$p^{p+1}$.
\end{proof}

\begin{remarknu}\label{rem:min}
Minimal polynomials are highly non-unique: for example, 
the polynomial
$$
(t^p-t+z)^2+p(t^p-t+z)
$$
is obviously also a minimal polynomial for the modulus $p^2$.
\end{remarknu}

Based on the observations in Proposition~\ref{prop:minpol}
and Lemma~\ref{lem:minpol}, we propose the following conjecture.

\begin{conjecture} \label{conj:1}
The degree of a minimal polynomial for the modulus $p^\ga$, $\ga\ge1$,
is the least $d$ such that $v_p(d!)\ge\ga$.
\end{conjecture} 

\begin{remarknu} \label{rem:1}
(1) Given the $p$-ary expansion of $d$, say
$$d=d_0+d_1\cdot p+d_2\cdot
p^2+\cdots+d_r\cdot p^r, \quad 0\le d_i\le p-1,$$ 
by the well-known formula of
Legendre \cite[p.~10]{LegeAA}, we have
\begin{align}\notag
v_p(d!)&=\sum _{\ell=1} ^{\infty}\fl{\frac {d} {p^\ell}}=
\sum _{\ell=1} ^{\infty}\fl{\sum _{i=0} ^{r}d_i p^{i-\ell}}=
\sum _{\ell=1} ^{\infty}\sum _{i=\ell} ^{r}d_i p^{i-\ell}\\
\label{eq:Leg}
&=
\sum _{i=1} ^{r}\sum _{\ell=1} ^{i}d_i p^{i-\ell}=
\sum _{i=1} ^{r}d_i\frac {p^{i}-1} {p-1}=
\frac {d-s(d)} {p-1},
\end{align}
where $s(d)$ denotes the sum of digits of $d$ in its $p$-ary
expansion.
Consequently, an equivalent way of phrasing Conjecture~\ref{conj:1}
is to say that the degree of a minimal polynomial for the 
modulus $p^\ga$ is the least $d$ with $d-s(d)\ge(p-1)\ga$.

\medskip
(2)
We claim that, in order to establish Conjecture~\ref{conj:1}, 
it suffices to prove the conjecture for $\ga=\frac {p^{\de}-1} {p-1}$, 
$\de=1,2,\dots$. If we take into account
Lemma~\ref{lem:minpol} plus the above remark, this means that it is
sufficient to prove that, for each $\de\ge1$, there is
a polynomial $A_\de(z,t)$ of degree $p^{\de}$ such that
\begin{equation} \label{eq:minpga}
A_\de(z,\Phi(z))=0\quad \text {modulo }p^{(p^\de-1)/(p-1)}.
\end{equation}
For, arguing by induction, let us suppose that 
we have already constructed $A_1(z,t),\break
A_2(z,t),
\dots,A_m(z,t)$ satisfying \eqref{eq:minpga}.
Let $\al$ be a positive integer which is divisible by~$p$ and
$$\al=\al_1\cdot p+\al_2\cdot
p^2+\cdots+\al_m\cdot p^m, \quad 0\le \al_i\le p-1,$$ 
be its $p$-ary expansion.
Then we have
\begin{equation} \label{eq:betaind}
\prod _{\de=1} ^{m}A_{\de}^{\al_\de}(z,\Phi(z))=0\quad \text {modulo }
\prod _{\de=1} ^{m}p^{\al_\de(p^\de-1)/(p-1)}=
p^{(\al-s(\al))/(p-1)}.
\end{equation}
On the other hand, the degree of the left-hand side of
\eqref{eq:betaind} as a polynomial in $\Phi(z)$ is
$\sum _{\de=1} ^{m}\al_\de p^\de=\al$.

Let us put these observations together.
In view of \eqref{eq:Leg},
Lemma~\ref{lem:minpol} says that the degree of a minimal
polynomial for the modulus $p^\ga$ cannot be smaller than the
least integer, $d^{(\ga)}$ say, for which $d^{(\ga)}-s(d^{(\ga)})\ge(p-1)\ga$. 
(We point out that $d^{(\ga)}$
must be automatically divisible by~$p$.) If we take into account that
the quantity $\al-s(\al)$, as a function in $\al$, is weakly monotone
increasing in $\al$, then \eqref{eq:betaind} tells us that,
as long as 
$$d^{(\ga)}\le (p-1)p+(p-1)p^2+\dots+(p-1)p^m=p^{m+1}-p,$$
we have found a monic
polynomial of degree $d^{(\ga)}$, $B_\ga(z,t)$ say, for which
$B_\ga(z,\Phi(z))=0$ modulo~$p^\ga$, namely the left-hand side
of \eqref{eq:betaind} with $\al$ replaced by $d^{(\ga)}$, to wit
$$
B_\ga(z,t)=\prod _{\de=1} ^{m}A_{\de}^{d^{(\ga)}_\de}(z,t),
$$
where $d^{(\ga)}=d^{(\ga)}_1\cdot p+d^{(\ga)}_2\cdot
p^2+\cdots+d^{(\ga)}_m\cdot p^m$ is the $p$-ary expansion of 
$d^{(\ga)}$. Hence, it must necessarily be a minimal polynomial
for the modulus $p^\ga$.

Since 
$$p^{m+1}-p
-s(p^{m+1}-p)=p^{m+1}-p-(p-1)m,$$ 
we have thus found minimal
polynomials for all moduli $p^\ga$ with 
$\ga\le \frac {p^{m+1}-p} {p-1}-m$. 
Now we should note that the
quantity $\al-s(\al)$ makes a jump from $p^{m+1}-p-(p-1)m$ to $p^{m+1}-1$
when we move from $\al=p^{m+1}-p$ to $\al=p^{m+1}$ (the reader should 
recall that it suffices to consider $\al$'s which are divisible by~$p$). 
If we take $A_m^p(z,t)$, which has degree $p\cdot p^m=p^{m+1}$, 
then, by \eqref{eq:minpga}, 
we also have a minimal polynomial for the modulus 
$$\left(p^{(p^m-1)/(p-1)}\right)^p=p^{(p^{m+1}-p)/(p-1)}$$
and, in view of the preceding remark, as well for all moduli $p^\ga$
with $\ga$ between\break 
$\frac {p^{m+1}-p} {p-1}-m+1$ and $\frac
{p^{m+1}-p} {p-1}$. 

So, indeed, the first modulus for which we do not have a minimal
polynomial is the modulus $p^{(p^{m+1}-p)/(p-1)+1}=
p^{(p^{m+1}-1)/(p-1)}$. This is the role
which $A_{m+1}(z,t)$ (see \eqref{eq:minpga} with $m+1$ in place of 
$\de$) would have to play.

\smallskip
The arguments above show at the same time that, supposing that we have
already constructed $A_1(z,t),A_2(z,t),\dots,A_m(z,t)$, the polynomial 
$A^p_m(z,t)$ is a very close ``approximation" to the polynomial
$A_{m+1}(z,t)$ that we are actually looking for next,
which is only ``off" by a factor of $p$.
In practice, one can recursively compute polynomials $A_\de(z,t)$
satisfying \eqref{eq:minpga} by following the procedure outlined in
the next-to-last paragraph before Lemma~\ref{lem:aiodd} in the next
section. It is these computations (part of which are reported in
Proposition~\ref{prop:minpol}) which have led us to believe in the truth of 
Conjecture~\ref{conj:1}.
\end{remarknu}

\section{Coefficient extraction from powers of $\Phi(z)$}
\label{sec:extr}

In the next section we are going to describe a method for expressing
formal power series satisfying certain differential equations, after
the coefficients of the series have been reduced modulo $p^k$, 
as polynomials in the
$p$-power series $\Phi(z)$ (which has been discussed in the previous
section; for the definition see \eqref{eq:Phidef}), the
coefficients being Laurent polynomials in $z$. Such a method would be
without value if we could not, at the same time, provide a procedure
for extracting coefficients from powers of $\Phi(z)$. The description
of such a procedure is the topic of this section.

Clearly, a brute force expansion of a power $\Phi^K(z)$, where $K$ is
a given positive integer, yields
\begin{equation} \label{eq:Phipot}
\Phi^K(z)=\sum _{r=1} ^{K}
\underset{b_1+\dots+b_r=K}{\sum _{b_1,\dots,b_r\ge1} ^{}}\frac {K!}
{b_1!\,b_2!\cdots b_r!}H_{b_1,b_2,\dots,b_r}(z),
\end{equation}
where
$$
H_{b_1,b_2,\dots,b_r}(z):=\sum_{n_1>n_2>\dots>n_r\ge0}
z^{b_1p^{n_1}+b_2p^{n_2}+\dots+b_rp^{n_r}}.
$$
The expansion \eqref{eq:Phipot} is not (yet) suited for our purpose,
since, when $b_1,b_2,\dots,b_r$ vary over all possible choices such
that their sum is $K$, the series $H_{b_1,b_2,\dots,b_r}(z)$ are {\it not\/}
linearly independent over the ring $\Z[z,z^{-1}]$ 
of Laurent polynomials in $z$ over the integers\footnote{The
same is true for an arbitrary commutative unital 
ring in place of the ring $\Z$ of
integers.},
and, secondly, coefficient extraction from a series
$H_{b_1,b_2,\dots,b_r}(z)$ can be a hairy task if some of the $b_i$'s 
are divisible by~$p$.

However, we shall show (see Corollary~\ref{lem:Hind}) that, 
if we restrict to $b_i$'s which are {\it relatively prime to~$p$}, 
then the corresponding series $H_{b_1,b_2,\dots,b_r}(z)$,
together with the (trivial) series~$1$, 
{\it are} linearly independent over $\Z[z,z^{-1}]$, and there is an
efficient algorithm to express an arbitrary 
series $H_{b_1,b_2,\dots,b_s}(z)$, without any
restriction on the $b_i$'s, as a linear combination over
$\Z[z,z^{-1}]$ of $1$ and the former series (see Proposition~\ref{conj:H}). 
Since coefficient
extraction from a series $H_{b_1,b_2,\dots,b_r}(z)$ with all $b_i$'s
relatively prime to~$p$ 
is straightforward (see Remark~\ref{rem:eff}), this solves the problem of
coefficient extraction from powers of $\Phi(z)$.

As a side result, the procedure which we described in the previous
paragraph, and which will be substantiated below, in combination
with \eqref{eq:Phipot}, provides all the
means for determining minimal polynomials in the sense of
Section~\ref{sec:Phi}: as explained in Item~(2) of Remark~\ref{rem:1} 
at the end of that section, it suffices to find a minimal polynomial
for the modulus $p^{(p^\de-1)/(p-1)}$, $\de=1,2,\dots$. For doing this, we 
would take a minimal polynomial $A_{\de-1}(z,t)$ for the modulus
$p^{(p^{\de-1}-1)/(p-1)}$, expand $A^p_{\de-1}(z,t)$,
and replace each coefficient $c_{\al,\be}$ of a monomial
$z^\alpha t^\beta$ in $A^p_{\de-1}(z,t)$ by 
$c_{\al,\be}+p^{(p^{\de}-p)/(p-1)}x_{\al,\be}$, 
where $x_{\al,\be}$ is a variable, thereby obtaining a modified polynomial, 
$B_{\de-1}(z,t)$ say.
Now we would substitute $\Phi(z)$
for~$t$, so that we obtain $B_{\de-1}(z,\Phi(z))$.
Here, we express powers of $\Phi(z)$
in terms of the series $H_{a_1,a_2,\dots,a_r}(z)$ with all $a_i$'s
relatively prime to~$p$, and collect terms.
By reading the coefficients of $z^\ga H_{a_1,a_2,\dots,a_r}(z)$ in this
expansion of $B_{\de-1}(z,\Phi(z))$ and equating them to zero 
modulo $p^{(p^\de-1)/(p-1)}$,
we produce a system of linear equations
modulo $p^{(p^\de-1)/(p-1)}$ in the unknowns $x_{\al,\be}$. 
By the definition of $A_{\de-1}(z,t)$, after division by $p^{(p^\de-p)/(p-1)}$,
this system reduces to a system
modulo $p$, that is, to a linear system of equations over the field
with $p$~elements. A priori, this system need not have a solution,
but experience seems to indicate that it always does; see
Conjecture~\ref{conj:1}.

\medskip
We start with an auxiliary result pertaining to the uniqueness of
representations of integers as sums of powers of $p$ with
multiplicities, tailor-made for application to the series
$H_{a_1,a_2,\dots,a_r}(z)$.

\begin{lemma} \label{lem:aiodd}
Let $d,r,s$ be positive integers with $r\ge s,$ $c$ an integer
with $\vert c\vert\le d,$ and let 
$a_1,a_2,\dots,a_r$ respectively $b_1,b_2,\dots,b_s$ be two sequences of 
integers, none of the $a_i$'s or $b_i$'s divisible by~$p,$ 
with $1\le a_i\le d$ for $1\le i\le r,$ and $1\le b_i\le
d$ for $1\le i\le s$. If
\begin{equation} \label{eq:a2b2}
a_1p^{2rd}+a_2p^{2(r-1)d}+\dots+a_rp^{2d}=
b_1p^{n_1}+b_2p^{n_2}+\dots+b_sp^{n_s}+c
\end{equation}
for integers $n_1,n_2,\dots,n_s$ with $n_1>n_2>\dots>n_s\ge0,$ then
$r=s,$ $c=0,$ $a_i=b_i,$ and $n_i=2d(r+1-i)$ for $i=1,2,\dots,r$.
\end{lemma}

\begin{proof}
We use induction on $r$.

First, let $r=1$. Then $s=1$ as well, and \eqref{eq:a2b2} becomes
\begin{equation} \label{eq:a2b2A}
a_1p^{2d}=b_1p^{n_1}+c.
\end{equation}
If $n_1>2d$, then the above equation implies
\begin{equation} \label{eq:a1c} 
a_1p^{2d}\equiv c\quad \text {modulo }p^{2d+1}.
\end{equation}
Consequently, the integer $c$ must be divisible by $p^{2d}$.
By assumption, we have $\vert c\vert\le d$. Since $d<p^{2d}$,
the only possibility is that $c$ vanishes. But then it follows from 
\eqref{eq:a1c} that $a_1$ is divisible by~$p$, a contradiction.

If $d<n_1< 2d$, then it follows from \eqref{eq:a2b2A} that $c$ must
be divisible by $p^{n_1}$. Again by assumption, we have 
$\vert c\vert\le d<p^d<p^{n_1}$, so that $c=0$. But then
\eqref{eq:a2b2A} cannot be satisfied since $b_1$ is
assumed not to be divisible by~$p$.

If $0\le n_1\le d$, then we estimate
$$
b_1p^{n_1}+c\le d\left(p^{d}+1\right)\le (p^d-1)(p^d+1)<p^{2d},
$$
which is again a contradiction to \eqref{eq:a2b2A}. 

The only remaining possibility is $n_1=2d$. If this is substituted in
\eqref{eq:a2b2A} and the resulting equation is combined with 
$\vert c\vert\le d<p^{2d}$, then the conclusion is that the equation
can only be satisfied if $c=0$ and $a_1=b_1$, in accordance with
the assertion of the lemma.

\medskip
We now perform the induction step. We assume that the assertion of
the lemma is established for all $r<R$, and we want to show that this
implies its validity for $r=R$. Let $t$ be maximal such that $n_t\ge
2d$. Then reduction of \eqref{eq:a2b2} modulo $p^{2d}$ yields
\begin{equation} \label{eq:a2b2B}
b_{t+1}p^{n_{t+1}}+b_{t+2}p^{n_{t+2}}+\dots+b_sp^{n_s}+c\equiv 0\quad 
\text {modulo }p^{2d}.
\end{equation}
Let us write $b\cdot p^{2d}$ for the left-hand side in
\eqref{eq:a2b2B}. Then, by dividing \eqref{eq:a2b2} (with $R$ instead
of $r$) by $p^{2d}$, we obtain
\begin{equation} \label{eq:a2b2C}
a_1p^{2(R-1)d}+a_2p^{2(R-2)d}+\dots+a_{R-1}p^{2d}=
b_1p^{n_1-2d}+b_2p^{n_2-2d}+\dots+b_tp^{n_t-2d}+b-a_R.
\end{equation}
We have
\begin{multline*}
0\le b\le
p^{-2d}d\left(p^{2d-1}+p^{2d-2}+\dots+p^{2d-s+t}+1\right)\\
\le p^{-2d}d\left(\tfrac {1} {p-1}(p^{2d}-p^{2d-s+t})+1\right)\le d.
\end{multline*}
Consequently, we also have $\vert b-a_R\vert\le d$.
This means that we are in a position to apply the induction
hypothesis to \eqref{eq:a2b2C}. The conclusion is that
$t=R-1$, $b-a_R=0$, $a_i=b_i$, and $n_i=2d(R+1-i)$ for
$i=1,2,\dots,R-1$. If this is used in \eqref{eq:a2b2} with $r=R$,
then we obtain
$$
a_Rp^{2d}=c
$$
or
$$
a_Rp^{2d}=b_Rp^{n_R}+c,
$$
depending on whether $s=R-1$ or $s=R$. The first case is absurd, since
$c\le d<p^{2d}\le a_Rp^{2d}$. On the other hand, the second case
has already been
considered in \eqref{eq:a2b2A}, and we have seen there that it
follows that $c=0$, $a_R=b_R$, and $n_R=2d$.

This completes the proof of the lemma.
\end{proof}

The announced independence of the series $H_{a_1,a_2,\dots,a_r}(z)$
with all $a_i$'s relatively prime to~$p$ is now an easy consequence.

\begin{corollary} \label{lem:Hind}
For any commutative unital ring $R,$
the series $H_{a_1,a_2,\dots,a_r}(z),$ with all $a_i$'s relatively
prime to~$p,$ together with the series $1$ are
linearly independent over $R[z,z^{-1}]$. In particular,
they are linearly independent over $(\Z/p\Z)[z,z^{-1}],$ 
over $(\Z/p^\ga\Z)[z,z^{-1}]$ for an arbitrary positive integer
$\ga,$ and over $\Z[z,z^{-1}]$.
\end{corollary}

\begin{proof}
Let us suppose that
\begin{equation} \label{eq:Hlincomb}
q_0(z)+\sum _{i=1}
^{N}q_i(z)H_{a_1^{(i)},a_2^{(i)},\dots,a_{r_i}^{(i)}}(z)=0,
\end{equation} 
where the $q_i(z)$'s are non-zero Laurent polynomials in $z$
over $R$, the $r_i$'s are positive integers, and 
$a_j^{(i)}$, $j=1,2,\dots,r_i$, $i=1,2,\dots,N$, are integers
relatively prime to~$p$. 
We may also assume that the tuples
$(a_1^{(i)},a_2^{(i)},\dots,a_{r_i}^{(i)})$, $i=1,2,\dots,N$, 
are pairwise distinct. Choose $i_0$ such
that $r_{i_0}$ is maximal among the $r_i$'s. Without loss of
generality, we may assume that the coefficient of $z^0$ in $q_{i_0}(z)$
is non-zero (otherwise we could multiply both sides of 
\eqref{eq:Hlincomb} by an appropriate power of $z$). Let $d$ be the
maximum of all $a_j^{(i)}$'s and the absolute values of exponents of
$z$ appearing in monomials with non-zero coefficient in
the Laurent polynomials $q_i(z)$, $i=0,1,\dots,N$. Then, according to
Lemma~\ref{lem:aiodd} with $r=r_{i_0}$, $a_j=a_j^{(i_0)}$,
$j=1,2,\dots,r_{i_0}$, the coefficient of
$$
z^{a_1^{(i_0)}p^{2rd}+a_2^{(i_0)}p^{2(r-1)d}+\dots+a_r^{(i_0)}p^{2d}}
$$
is $1$ in $H_{a_1^{(i_0)},a_2^{(i_0)},\dots,a_{r_{i_0}}^{(i_0)}}(z)$,
while it is zero in series 
$z^eH_{a_1^{(i_0)},a_2^{(i_0)},\dots,a_{r_{i_0}}^{(i_0)}}(z)$, where
$e$ is a non-zero integer with $\vert e\vert\le d$, and in all other series
$z^eH_{a_1^{(i)},a_2^{(i)},\dots,a_{r_i}^{(i)}}(z)$,
$i=1,\dots,\break i_0-1,i_0+1,\dots,N$, where $e$ is a (not necessarily
non-zero) integer with $\vert e\vert\le d$.
This contradiction to \eqref{eq:Hlincomb} establishes
the assertion of the corollary.
\end{proof}

\begin{remarknu} \label{rem:eff}
Coefficient extraction from a series $H_{a_1,a_2,\dots,a_{r}}(z)$
with all $a_i$'s relatively prime to~$p$ 
is straightforward: if we want to know whether
$z^M$ appears in $H_{a_1,a_2,\dots,a_{r}}(z)$, that is, whether we can
represent $M$ as
$$
M=a_1p^{n_1}+a_2p^{n_2}+\dots+a_rp^{n_r}
$$
for some $n_1,n_2,\dots,n_r$ with $n_1>n_2>\dots>n_r\ge0$, then
necessarily $n_r=v_p(M)$, $n_{r-1}=v_p(M-a_rp^{n_r})$, etc.
The term $z^M$ appears in $H_{a_1,a_2,\dots,a_{r}}(z)$ if, and only
if, the above process terminates after {\it exactly} $r$ steps.
This means, that, with $n_r,n_{r-1},\dots,n_1$ constructed as above, 
we have
$$
M-\left(a_sp^{n_s}+\dots+a_{r-1}p^{n_{r-1}}+a_rp^{n_r}\right)>0
$$
for $s>1$, and
$$
M-\left(a_1p^{n_1}+\dots+a_{r-1}p^{n_{r-1}}+a_rp^{n_r}\right)=0.
$$
It should be noted that, given $a_1,a_2,\dots,a_r$, this procedure of
coefficient extraction needs at most $O(\log M)$ operations, that is,
its computational complexity is linear.
\end{remarknu}

It remains to show that an arbitrary series $H_{b_1,b_2,\dots,b_s}(z)$
can be expressed as a linear combination over $\Z[z,z^{-1}]$ of
the series $1$ and 
the series $H_{a_1,a_2,\dots,a_r}(z)$, where all $a_i$'s are 
relatively prime to~$p$.

\begin{proposition} \label{conj:H}
For any positive integers $b_1,b_2,\dots,b_r,$ the series
$H_{b_1,b_2,\dots,b_r}(z)$
can be expressed as a linear combination
over $\Z[z]$ of the series $1$ and series of the form
$H_{a_1,a_2,\dots,a_s}(z),$ where all $a_i$'s are 
relatively prime to~$p,$ $s\le r,$ and
$a_1+a_2+\dots+a_s\le b_1+b_2+\dots+b_r$.
\end{proposition}

\begin{proof}
Given $H_{b_1,b_2,\dots,b_r}(z)$
where not all $b_i$'s are relatively
prime to~$p$, we let $h$ be maximal
such that $b_h\equiv0$~(mod~$p$). If there is no such~$h$, then
there is nothing to do. If there is,
then we apply the appropriate case of
\begin{multline} \label{eq:Rek}
H_{b_1,\dots,b_{h-1},pb'_h,b_{h+1},\dots,b_r}(z)
\\=
\begin{cases} 
H_{b_1,\dots,b_{h-1},b'_h,b_{h+1},\dots,b_r}(z)
+H_{b_1,\dots,b_{h-1}+b'_h,b_{h+1},\dots,b_r}(z)
-H_{b_1,\dots,b_{h-1},pb'_h+b_{h+1},\dots,b_r}(z),\\
&\hskip-2.5cm\text{if $1<h<r$,}\\
H_{b'_1,b_2,\dots,b_r}(z)
-H_{pb'_1+b_2,\dots,b_r}(z),&\hskip-2.5cm\text{if $1=h<r$,}\\
H_{b_1,\dots,b_{h-1},b'_h}(z)+H_{b_1,\dots,b_{h-1}+b'_h}(z)
-z^{b'_h}H_{b_1,\dots,b_{h-1}}(z),&\hskip-2.5cm\text{if $1<h=r$,}\\
H_{b'_1}(z)-z^{b'_1},&\hskip-2.5cm\text{if }1=h=r.
\end{cases}
\end{multline}
These identities are due to Hou \cite{HouQAA}, and they
are straightforward to verify. The reduction \eqref{eq:Rek}
is now applied to the series which arose during the first
application of \eqref{eq:Rek}. This process is repeated until
all series $H_{a_1,a_2,\dots,a_s}(z)$ in the obtained expression
have the property that each index $a_i$ is relatively prime to~$p$.
From \eqref{eq:Rek}, it is obvious that all these series satisfy
$a_1+a_2+\dots+a_s\le b_1+b_2+\dots+b_r$.
This proves the assertions of the proposition.
\end{proof}

To conclude this section, let us provide an illustration of the above
discussion. We set ourselves the task of determining the coefficient
of $z^{297398301914493}$ in $\Phi^5(z)$ for $\Phi(z)=\sum_{n\ge0}z^{3^n}$.
In order to accomplish this task, we first express $\Phi^5(z)$ in
terms of series $H_{a_1,\dots,a_r}(z)$ with all $a_i$'s being 
relatively prime to~$3$. 
Indeed, by means of the expansion \eqref{eq:Phipot} and 
the algorithm described in the proof of Proposition~\ref{conj:H},
we have
\begin{multline}
\Phi^5(z)=
 120 H_{1, 1, 1, 1, 1}(z)
+ 60 H_{2, 1, 1, 1} (z)
+ 60 H_{1, 2, 1, 1} (z)
+ 60 H_{1, 1, 2, 1} (z)
+ 60 H_{1, 1, 1, 2} (z)\\
+ 30 H_{2, 2, 1} (z)
+ 30 H_{2, 1, 2} (z)
+ 30 H_{1, 2, 2} (z)
- 15 H_{4, 1} (z)
- 15 H_{1, 4} (z)
- 9 H_{5} (z)\\
+ 60 H_{1, 1, 1}(z) 
+ 30 H_{2, 1} (z)
+ 30 H_{1, 2} (z)
- 20 z H_{1, 1} (z)
- 10 z H_{2} (z)
+ 10 H_{1} (z)
-10 z .
\label{eq:Phi35}
\end{multline}

Now we have to answer the question, in which of the series
$H_{a_1,\dots,a_r}(z)$ that appear in this expansion of $\Phi^5(z)$
do we find the monomial $z^{297398301914493}$. Using the algorithm
described in Remark~\ref{rem:eff}, we see that 
{\allowdisplaybreaks
\begin{align*}
297398301914493
&=3^{30}+3^{29}+3^{28},\\
&=2\cdot3^{29}+2\cdot3^{29}+3^{28},\\
&=2\cdot3^{28}+3^{ 28}+2\cdot3^{ 28}+ 183014339639688,\\
&=3^{30}+2\cdot3^{ 28}+2\cdot3^{ 28},\\
&=4\cdot3^{29}+3^{ 28},\\
&=3^{30}+4\cdot3^{ 28},\\
&=3^{29}+3^{ 29}+2\cdot3^{ 29}+3^{ 28},\\
&=3^{30}+3^{ 28}+3^{ 28}+2\cdot3{ 28},\\
&=5\cdot3^{28}+183014339639688,\\
&=2\cdot3^{29}+3^{ 28}+137260754729766,\\
&=3^{28}+2\cdot{ 28}+228767924549610,\\
&=2\cdot3^{28}+251644717004571,\\
&=3^{28}+274521509459532.
\end{align*}}%
Here, the first line shows that $z^{297398301914493}$ appears in
$H_{1,1,1}(z)$ (thereby making it impossible to appear
in $H_{1,1,1,1,1}(z)$, $H_{2,1,1,1}(z)$, or $H_{1,2,1,1}(z)$), 
the fifth line shows that it appears in
$H_{4,1}(z)$, and the sixth line shows that it appears in
$H_{1,4}(z)$, while the remaining lines show that it does not
appear in any other term in the expansion of $\Phi^5(z)$ displayed in
\eqref{eq:Phi35}. Hence, by taking into account the coefficients
with which the series $H_{1,1,1}(z)$, $H_{4,1}(z)$, and $H_{1,4}(z)$ appear in
this expansion, the coefficient of $z^{297398301914493}$
in $\Phi^5(z)$ is seen to equal $60-15-15=30$.

\section{The method}
\label{sec:method}

We consider a (formal) differential equation 
\begin{equation} \label{eq:diffeq}
\mathcal P(z;F(z),F'(z),F''(z),\dots,F^{(s)}(z))=0,
\end{equation}
where $\mathcal P$ is a polynomial with integer coefficients, which has a
power series solution $F(z)$ with integer coefficients. 
We assume that this power series solution $F(z)$ is uniquely
determined by
\eqref{eq:diffeq} over the integers, and also
over all rings $(\Z/p^\ga\Z)$, where $p$ is a fixed prime
number and $\ga$ any positive integer.

In the above
situation, we propose the following algorithmic approach to
determining the series $F(z)$ modulo a $p$-power $p^{p^\al}$,
for some non-negative integer $\al$. We make the Ansatz
\begin{equation} \label{eq:Ansatz}
F(z)=\sum _{i=0} ^{p^{\al+1}-1}a_i(z)\Phi^i(z)\quad \text {modulo
}p^{p^\al},
\end{equation}
with $\Phi(z)$ as given in \eqref{eq:Phidef}, and where the $a_i(z)$'s
are (at this point) undetermined Laurent polynomials in $z$.
Now we substitute \eqref{eq:Ansatz} into \eqref{eq:diffeq}, and
we shall gradually determine approximations $a_{i,\be}(z)$ to $a_i(z)$ such that
\eqref{eq:diffeq} holds modulo $p^\be$, for $\be=1,2,\dots,p^{\al}$. 
To start the procedure, we consider the differential equation
\eqref{eq:diffeq} modulo $p$, with
\begin{equation} \label{eq:Ansatz1}
F(z)=F_1(z)=\sum _{i=0} ^{p^{\al+1}-1}a_{i,1}(z)\Phi^i(z)\quad \text {modulo
}p.
\end{equation}
Using the elementary fact that $\Phi'(z)=1$ modulo $p$, we see that
the left-hand side of \eqref{eq:diffeq} is a polynomial in
$\Phi(z)$ with coefficients that are Laurent polynomials in $z$. 
We reduce powers $\Phi^k(z)$ with $k\ge p^{\al+1}$ using
the relation (which is implied by the minimal polynomial for the
modulus~$p$ given in Proposition~\ref{prop:minpol})\,\footnote{Actually, 
if we would like to obtain an optimal result, 
we should use the relation implied by
a minimal polynomial for the modulus $p^{p^\al}$ in the sense of 
Section~\ref{sec:Phi}. But since we have no general formula available
for such a minimal polynomial (cf.\ Item (2) of Remark~\ref{rem:1} in that section),
and since we wish to prove results for arbitrary moduli, choosing instead
powers of a minimal polynomial for the modulus $p$ is the best compromise.
In principle, it may happen that there exists a polynomial in $\Phi(z)$
with coefficients that are Laurent polynomials in $z$, which is identical
with $F(z)$ after reduction of its coefficients modulo $p^{p^\al}$,
but the Ansatz \eqref{eq:Ansatz} combined with the
reduction \eqref{eq:PhiRel} fails because it is too restrictive. 
We are not aware of a concrete example where this obstruction occurs.
The subgroup numbers of $SL_2(\Z)$ which are treated
modulo~$8$ in \cite[Section~11]{KaKMAA} by 
the method described here (specialised to $p=2$), 
and modulo~$16$ by an enhancement of the method
outlined in \cite[Appendix~D]{KaKMAA}
are a potential candidate when
considered modulo $2^{\be}$ for $\be\ge5$.
On the other hand, once we are successful using this
(potentially problematic) Ansatz, then the result can easily be converted into
an optimal one by further reducing the polynomial thus obtained, using
the relation implied by a minimal polynomial for the modulus
$p^{p^\al}$.}
\begin{equation} \label{eq:PhiRel}
\big(\Phi^p(z)-\Phi(z)+z\big)
^{p^\al}=0\quad \text {modulo }p^{p^\al}.
\end{equation}
Since, at this point, we are only interested in finding a solution to
\eqref{eq:diffeq} modulo~$p$, the above relation simplifies to
\begin{equation} \label{eq:PhiRel2}
\Phi^{p^{\al+1}}(z)-\Phi^{p^{\al}}(z)+z^{p^{\al}}=0\quad 
\text{modulo }p.
\end{equation}
Now we compare coefficients of powers $\Phi^k(z)$,
$k=0,1,\dots,p^{\al+1}-1$ (see Remark~\ref{rem:comp}). This yields a
system of $p^{\al+1}$ (differential) equations (modulo~$p$)
for the unknown Laurent polynomials $a_{i,1}(z)$, $i=0,1,\dots,p^{\al+1}-1$,
which may or may not have a solution. 

Provided we have already found
Laurent polynomials $a_{i,\be}(z)$, $i=0,1,\dots,p^{\al+1}-1$, for some $\be$
with $1\le \be\le p^{\al}-1$, such that
\begin{equation} \label{eq:Ansatz2}
\sum _{i=0} ^{p^{\al+1}-1}a_{i,\be}(z)\Phi^i(z)
\end{equation}
solves \eqref{eq:diffeq} modulo~$p^\be$, we put 
\begin{equation} \label{eq:Ansatz2a}
a_{i,\be+1}(z):=a_{i,\be}(z)+p^{\be}b_{i,\be+1}(z),\quad 
i=0,1,\dots,p^{\al+1}-1,
\end{equation}
where the $b_{i,\be+1}(z)$'s are (at this point) undetermined 
Laurent polynomials in $z$. Next we substitute
\begin{equation} \label{eq:Ansatz2b}
\sum _{i=0} ^{p^{\al+1}-1}a_{i,\be+1}(z)\Phi^i(z)
\end{equation}
for $F(z)$ in \eqref{eq:diffeq}. Using the fact that
$\Phi'(z)=\sum _{n=0} ^{\be}p^nz^{p^n-1}$ modulo~$p^{\be+1}$, 
we expand the left-hand side as a polynomial in $\Phi(z)$ (with
coefficients being Laurent polynomials in $z$), we apply again the reduction
using relation \eqref{eq:PhiRel}, we compare coefficients of powers
$\Phi^k(z)$, $k=0,1,\dots,p^{\al+1}-1$ (again, see
Remark~\ref{rem:comp}),
and, as a result, we obtain a
system of $p^{\al+1}$ (differential) equations (modulo~$p^{\be+1}$)
for the unknown Laurent polynomials $b_{i,\be+1}(z)$, $i=0,1,\dots,p^{\al+1}-1$,
which may or may not have a solution. If we manage to push this 
procedure through until $\be=p^{\al}-1$, then,
setting $a_i(z)=a_{i,p^{\al}}(z)$, $i=0,1,\dots,p^{\al+1}-1$,
the right-hand side of \eqref{eq:Ansatz} is a solution to
\eqref{eq:diffeq} modulo~$p^{p^\al}$, as required.

\begin{remarknu} \label{rem:comp}
As the reader will have noticed, each comparison of coefficients of
powers of $\Phi(z)$ is based on the ``hope" that, 
if a polynomial in $\Phi(z)$ is zero modulo a $p$-power $p^\be$
(as a formal Laurent series), then already all coefficients
of powers of $\Phi(z)$ in this polynomial vanish modulo $p^\be$.
However, this implication is false in general 
(see Lemma~39 in \cite[Appendix~D]{KaKMAA} for the case of modulus~$16$).
It may thus happen that the method described in this section fails to find a
solution modulo $p^\be$
to a given differential equation in the form of a polynomial in
$\Phi(z)$ with coefficients that are Laurent polynomials in $z$ over
the integers, while such a solution does in fact
exist. As a matter of fact, this situation already occurred
earlier in \cite[Theorem~28]{KaKMAA} in the context of 
the analysis modulo~$16$ of the
subgroup numbers of $SL_2(\Z)$.
There exists an enhancement of the method, outlined for the case of $p=2$
in \cite[Appendix~D]{KaKMAA},
which (at least in principle) allows us to decide whether
or not a solution modulo a given power in terms of a polynomial in
$\Phi(z)$ with coefficients that are Laurent polynomials in $z$ over
the integers exists, and, if so, to explicitly find such a solution.
Since we do not make use of this enhancement in the present article,
we refrain from presenting it here.
\end{remarknu}

It is not difficult to see that performing the iterative step 
\eqref{eq:Ansatz2a} amounts to solving a system of linear
differential equations in the unknown functions $b_{i,\be+1}(z)$
modulo~$p$, where all of them are Laurent polynomials in $z$,
and where only derivatives up to order $p-1$ of the $b_{i,\be+1}(z)$'s occur.
Solving such a system is equivalent to solving an ordinary system of
linear equations. This equivalence was explained for the case of
$p=2$ in \cite[Lemma~12]{KaKMAA}. Since we do not need the general
version of this lemma for arbitrary primes~$p$ in the present article,
we leave it to the reader to work out the straightforward details.


\medskip
We remark that the idea of the method that we have described in this
section has certainly further potential. For example, the fact that
the series $\Phi(z)$ remains invariant under the substitution $z\to
z^p$ (or, more generally, under the substitution $z\to z^{p^h}$, 
where $h$ is some
positive integer) --- up to a simple additive correction --- can be
exploited in order to extend the range of applicability of our method
to equations where we not only allow differentiation but also this
kind of substitution. This idea is actually already used in
Section~\ref{sec:FCat}, see Theorem~\ref{thm:FCat}.

\section{The number of non-crossing connected graphs modulo $3$-powers}
\label{sec:noncr}

In this section, we apply our method to determining congruences
modulo powers of~$3$ for the numbers of non-crossing connected graphs
with a given number of vertices. The origin of this analysis lies
in a conjecture of Deutsch and Sagan \cite[Conj.~5.18]{DeSaAA}.
See \eqref{eq:DeSaAA} and the subsequent paragraph for more
information.

A non-crossing graph on $n$ vertices is a graph whose vertices are
labelled by $1,2,\dots,n$, which can be embedded in the plane without
crossings of edges in such a manner that the vertices are placed
around a circle in clock-wise order.
Let $N_n$ denote the number of non-crossing connected graphs with $n$
vertices. It is known from \cite[Eq.~(16)]{GessAZ} that
\begin{equation} \label{eq:Nn} 
N_n=\frac {2^{2n-1}} {n}\binom {\frac {3} {2}n-2}{n-1}
-\frac {2^{2n-2}} {n}\binom {\frac {3} {2}n-\frac {3} {2}}{n-1}.
\end{equation}
Let $N(z):=\sum _{n\ge1} ^{}N_n\,z^n$ be the corresponding generating
function. According to
\cite[Eq.~(47)]{DB} (see also \cite[Sec.~2.1]{FN}), the generating function
$N(z)$ satisfies the polynomial equation
\begin{equation} \label{eq:fmdiff}
{N}^3(z) + {N}^2(z) - 3z\,{N}(z) + 2z^2 = 0.
\end{equation}
It is not difficult to see that \eqref{eq:fmdiff} determines $N(z)$
uniquely as a formal power series over $\Z$, and as well 
modulo any power of~$3$, provided $N_0=0$ and $N_1=1$.

\begin{theorem} \label{thm:N}
Let $\Phi(z)=\sum _{n\ge0} ^{}z^{3^n},$ and let $\al$ be a
non-negative integer.
Then the generating function $N(z),$ when reduced modulo $3^{3^\al},$ 
can be expressed as a polynomial in $\Phi(z)$ of degree at most
$3^{\al+1}-1,$ with coefficients that are Laurent polynomials in $z$
over the integers.
\end{theorem}

\begin{remark}
Computer experiments seem to indicate that the generating function
$N(z)$ in Theorem~\ref{thm:N}, when reduced modulo $3^{3^\al},$ 
can in fact be expressed as a polynomial in $\Phi(z)$ of degree at most
$3^{\al+1}-1,$ with coefficients that are {\it polynomials} in $z$
over the integers.
\end{remark}

\begin{proof}[Proof of Theorem~{\em\ref{thm:N}}]
We apply the method from Section~\ref{sec:method}. We start by
substituting the Ansatz \eqref{eq:Ansatz1} in \eqref{eq:fmdiff} and
reducing the result modulo $3$. 
In this way, we obtain
\begin{multline} \label{eq:noncr-1}
\sum _{i=0} ^{3^{\al+1}-1}a_{i,1}^3(z)\Phi^{3i}(z)+
\sum _{i=0} ^{3^{\al+1}-1}a_{i,1}^2(z)\Phi^{2i}(z)\\
-\sum _{0\le i<j\le 3^{\al+1}-1}a_{i,1}(z)a_{j,1}(z)\Phi^{i+j}(z)
+2z^2=0\quad \text
{modulo }3.
\end{multline}
We claim that the following choices solve the above congruence:
\begin{align} \notag
a_{0,1}(z)&=s^2_\al(z)+s_\al(z)
\quad \text {modulo }3,\\[2mm]
\notag
a_{3^{\al},1}(z)&=1-s_\al(z)
\quad \text {modulo }3,\\
a_{2\cdot 3^{\al},1}(z)&=1\quad \text {modulo }3,
\label{eq:baseN}
\end{align}
where $s_\al(z)=\sum _{k=0} ^{\al-1}z^{3^k}$,
with all other $a_{i,1}(z)$ vanishing.
In order to verify the claim, we substitute our choices in
\eqref{eq:noncr-1}. For the left-hand side, we obtain
\begin{multline*} 
\Phi^{2\cdot 3^{\al+1}}(z)
+\left(1-s^3_\al(z)\right)\Phi^{3^{\al+1}}(z)
+s^6_\al(z)+s^3_\al(z)\\
+\Phi^{4\cdot 3^{\al}}(z)
+\left(1+s^2_\al(z)+s_\al(z)\right)\Phi^{2\cdot 3^{\al}}(z)
+s^4_\al(z)-s^3_\al(z)+s^2_\al(z)
\\
-\left(1-s_\al(z)\right)\Phi^{3\cdot 3^{\al}}(z)
-\left(s^2_\al(z)+s_\al(z)\right)\Phi^{2\cdot 3^{\al}}(z)
-\left(s_\al(z)-s^3_\al(z)\right)\Phi^{3^{\al}}(z)
+2z^2\\
\quad \quad \text{modulo 3}.
\end{multline*}
Using the relation \eqref{eq:PhiRel2} with $p=3$
and reducing the obtained expression modulo~3,
we arrive at
\begin{multline*} 
\Phi^{2\cdot 3^{\al}}(z)-2z^{3^\al}\Phi^{3^{\al}}(z)+z^{2\cdot 3^\al}
+\left(s_\al(z)-s^3_\al(z)\right)
\left(\Phi^{3^{\al}}(z)-z^{3^\al}\right)
\\
+\Phi^{3^{\al}}(z)\left(\Phi^{3^{\al}}(z)-z^{3^\al}\right)
+\Phi^{2\cdot 3^{\al}}(z)
+s^6_\al(z)+s^4_\al(z)+s^2_\al(z)
\\
-\left(s_\al(z)-s^3_\al(z)\right)\Phi^{3^{\al}}(z)
+2z^2\\
=z^{2\cdot 3^\al}-z^{3^\al}s_\al(z)+z^{3^\al}s^3_\al(z)
+s^6_\al(z)+s^4_\al(z)+s^2_\al(z)-z^2
\quad \quad \text{modulo 3}.
\end{multline*}
By using the relation
\begin{equation} \label{eq:sal} 
s^3_\al(z)=s_\al(z)+z^{3^{\al}}-z
\quad \quad \text{modulo 3}
\end{equation}
several times, this expression can be turned into
\begin{multline*}
z^{2\cdot 3^\al}-z^{3^\al}s_\al(z)
+z^{3^\al}\left(s_\al(z)+z^{3^\al}-z\right)
+\left(s_\al(z)+z^{3^{\al}}-z\right)^2\\
+s_\al(z)\left(s_\al(z)+z^{3^{\al}}-z\right)+s^2_\al(z)-z^2
\quad \quad \text{modulo 3}.
\end{multline*}
After expansion and reduction modulo~$3$, one sees that this
expression reduces to zero.

\medskip
After we have completed the ``base  step," we now proceed with the
iterative steps described in Section~\ref{sec:method}. We consider
the Ansatz \eqref{eq:Ansatz2}--\eqref{eq:Ansatz2b}, where the
coefficients $a_{i,\be}(z)$ are supposed to provide a solution
$F_{\be}(z)=\sum _{i=0} ^{3^{\al+1}-1}a_{i,\be}(z)\Phi^i(z)$ to
\eqref{eq:fmdiff} modulo~$3^\be$. This Ansatz, substituted in
\eqref{eq:fmdiff}, produces the congruence
\begin{multline} \label{eq:ItGl}
-3^\be\sum _{j=0} ^{3^{\al+1}-1}\sum _{i=0} ^{3^{\al+1}-1}
a_{j,\be}(z)b_{i,\be+1}(z)\Phi^{i+j}(z)
+F^3_{\be}(z)
+F^2_{\be}(z)-3zF_{\be}(z)+2z^2=0\\
\quad \quad 
\text {modulo }3^{\be+1}.
\end{multline}
Since the sum has the prefactor $3^\be$, we may reduce the
$a_{i,\be}(z)$ modulo $3$. By construction, we have
\begin{align*}
a_{0,\be}(z)&=a_{0,1}(z)=s^2(\al)+s_\al(z)
\quad \quad \text{modulo }3,\\
a_{3^\al,\be}(z)&=a_{3^\al,1}(z)=1-s_\al(z)
\quad \quad \text{modulo }3,\\
a_{2\cdot3^\al,\be}(z)&=a_{2\cdot3^\al,1}(z)=1
\quad \quad \text{modulo }3,
\end{align*}
and $a_{i,\be}=0$ modulo 3 for all other $i$'s.
If we substitute this in \eqref{eq:ItGl} and subsequently use
\eqref{eq:PhiRel2} with $p=3$ to reduce high powers of $\Phi(z)$, we obtain
\begin{multline*} 
-3^\be
\left(s^2_\al(z)+s_\al(z)\right)
\sum _{i=0} ^{3^{\al+1}-1}
b_{i,\be+1}(z)\Phi^{i}(z)
-3^\be \left(1-s_\al(z)\right)
\sum _{i=0} ^{3^{\al+1}-1}
b_{i,\be+1}(z)\Phi^{3^\al+i}(z)\\
-3^\be\sum _{i=0} ^{3^{\al+1}-1}
b_{i,\be+1}(z)\Phi^{2\cdot3^\al+i}(z)
+F^3_{\be}(z)
+F^2_{\be}(z)-3zF_{\be}(z)+2z^2\\
=
-3^\be
\sum _{i=0} ^{3^{\al}-1}
\left(
\left(s^2_\al(z)+s_\al(z)\right)
b_{i,\be+1}(z)\right.
\kern8cm
\\
\kern3cm
\left.
-
\left(1-s_\al(z)\right)z^{3^\al}b_{i+2\cdot3^\al,\be+1}(z)
-z^{3^\al}b_{i+3^\al,\be+1}(z)
\right)
\Phi^{i}(z)\\
-3^\be
\sum _{i=3^\al} ^{2^\cdot3^{\al}-1}
\left(
\left(s^2_\al(z)+s_\al(z)\right)
b_{i,\be+1}(z)
+\left(1-s_\al(z)\right)b_{i-3^\al,\be+1}(z)\right.
\kern3cm
\\
\kern3cm
\left.
+\left(1-s_\al(z)\right)b_{i+3^\al,\be+1}(z)
+b_{i,\be+1}(z)
-z^{3^\al}b_{i+3^\al,\be+1}(z)
\right)
\Phi^{i}(z)\\
-3^\be
\sum _{i=2\cdot3^\al} ^{3^{\al+1}-1}
\left(
\left(s^2_\al(z)+s_\al(z)\right)
b_{i,\be+1}(z)
+\left(1-s_\al(z)\right)b_{i-3^\al,\be+1}(z)\right.
\kern3cm
\\
\kern3cm
\left.
+b_{i-2\cdot3^\al,\be+1}(z)
+b_{i,\be+1}(z)
\right)
\Phi^{i}(z)\\
+F^3_{\be}(z)
+F^2_{\be}(z)-3zF_{\be}(z)+2z^2=0
\quad \quad 
\text {modulo }3^{\be+1}.
\end{multline*}
By our assumption on $F_{\be}(z)$, we may divide by $3^\be$.
Comparison of powers of $\Phi(z)$ then yields a system of congruences
of the form
\begin{equation} \label{eq:Abc} 
M\cdot b=c\quad \text{modulo }3,
\end{equation}
where $b$ is the column vector of unknowns
$(b_{i,\be+1}(z))_{i=0,1,\dots,3^{\al+1}-1}$, $c$ is a (known)
column vector of Laurent polynomials in $z$, and $M$ is the matrix
$$
\begin{pmatrix} 
D(s^2_\al(z)+s_\al(z))&D(-z^{3^\al})&D(-z^{3^\al}(1-s_\al(z)))\\
D(1-s_\al(z))&D(s^2_\al(z)+s_\al(z)+1)&D(1-s_\al(z)-z^{3^\al})\\
D(1)&D(1-s_\al(z))&D(s^2_\al(z)+s_\al(z)+1)
\end{pmatrix},
$$
with $D(x)$ denoting the $3^\al\times3^\al$ diagonal matrix whose
diagonal entries equal $x$. We claim that
$$
\det (M)=z^{2\cdot 3^\al}\quad \quad \text{modulo }3.
$$
In order to see this, we subtract $s^2_\al(z)+s_\al(z)$ times row
$i+2\cdot3^\al$ from row $i$, $i=0,1,\dots,3^{\al}-1$,
and we subtract $1-s_\al(z)$ times row
$i+2\cdot3^\al$ from row $i+3^\al$, $i=0,1,\dots,3^{\al}-1$.
If, in addition, we move rows $2\cdot 3^\al,2\cdot 3^\al+1,
\dots,3^{\al+1}-1$ to the top of the matrix, then
it attains an upper triangular form, from which
one infers that
\begin{align*}
\det(M)&=(-1)^{2\cdot3^\al\cdot3^\al}
\det\!{}^2 \big(D(s^3_\al(z)-s_\al(z)-z^{3^\al})\big)
\det\big(D(1)\big)\\
&=\big(s^3_\al(z)-s_\al(z)-z^{3^\al}\big)^{2\cdot 3^\al}\\
&=z^{2\cdot 3^\al}\quad \quad \text{modulo }3,
\end{align*}
by means of \eqref{eq:sal}, as we claimed. As a consequence,
the system \eqref{eq:Abc} is
(uniquely) solvable. Thus, we have proved that, for an arbitrary non-negative
integer $\al$, the algorithm of
Section~\ref{sec:method} will produce a solution $F_{{3^{\al}}}(z)$ 
to \eqref{eq:fmdiff} modulo $3^{3^\al}$ which is a
polynomial in $\Phi(z)$ with coefficients that are Laurent polynomials in
$z$.
\end{proof}

We have implemented this algorithm. As an illustration, the next
theorem contains the result for the modulus $27$.

{
\allowdisplaybreaks
\begin{theorem} \label{thm:noncr27}
Let $\Phi(z)=\sum _{n\ge0} ^{}z^{3^n}$.
Then we have
\begin{multline} 
\label{eq:Loes1}
\sum _{n\ge1} ^{}N_n\,z^n
= 18 z^3 + z^2 + z
+ (18 z^3+12 z^2 ) \Phi(z) 
+ ( 3 z^2+15 z ) \Phi^2(   z)\\
+ ( 9 z^2 + 5 z+13 ) \Phi^3(z) 
+ ( 9 z^2+ 6 z +24) \Phi^4(  z) 
+ (15 z + 6 ) \Phi^5(z) \\
+ (18 z + 4 ) \Phi^6(z) 
+ (18 z + 21 ) \Phi^7(z) + 
 12 \Phi^8(z)
\quad \quad 
\text {\em modulo }27.
\end{multline}
\end{theorem}
}

In \cite[Conj.~5.18]{DeSaAA}, Deutsch and Sagan conjectured that
\begin{equation} \label{eq:DeSaAA} 
N_n\equiv\begin{cases} 
1\pmod 3&\text{if }n=3^i\text{ or }n=2\cdot 3^i\text{ for an integer 
$i\ge0$},\\
2\pmod 3&\text{if }n=3^{i_1}+3^{i_2}\text{ for integers $i_1>i_2\ge0$},\\
0\pmod 3&\text{otherwise.}
\end{cases}
\end{equation}
This conjecture was proved by Eu, Liu and Yeh \cite{EuLYAA} by
carefully analysing a binomial sum formula for $N_n$ modulo~$3$.
A simpler proof was given by Gessel in \cite{GessAZ}.
It should be noted that his proof is essentially the one resulting
from the $\al=0$ case of Theorem~\ref{thm:N}. Moreover,
our results allow
us to generalise this --- in principle --- to any power of~$3$.
For the sake of illustration, we display the results for the
moduli $9=3^2$ and $27=3^3$.

\begin{corollary} \label{cor:N27}
The numbers $N_n$ of connected non-crossing  graphs
obey the following congruences modulo $27${\em:}

\begin{enumerate}
\item[(i)] 
$N_n\equiv1$~{\em (mod~$27$)} if, and only if, 
$n=1$,
or
$n=2\cdot 3^i$ with $i\ge0;$

\item[(ii)] 
$N_n\equiv2$~{\em (mod~$27$)} if, and only if, 
$n=10\cdot 3^i$ with $i\ge1;$

\item[(iii)] 
$N_n\equiv3$~{\em (mod~$27$)} if, and only if, 
$n=2\cdot 3^i+1$ with $i\ge1$,
or
$n=3^i+2$ with $i\ge1$,
or
$n=2\cdot3^{i_1}+2\cdot3^{i_2}$ with $i_1>i_2\ge0$,
or
$n=4\cdot3^{i_1}+3^{i_2}+3^{i_3}$ with $i_1-2>i_2-1>i_3\ge0$,
or
$n=3^{i_1}+4\cdot3^{i_2}+3^{i_3}$ with $i_1-3>i_2-1>i_3\ge0$,
or
$n=3^{i_1}+3^{i_2}+4\cdot3^{i_3}$ with $i_1-3>i_2-2>i_3\ge0;$

\item[(iv)] 
$N_n\equiv4$~{\em (mod~$27$)} if, and only if, 
$n=3;$

\item[(v)] 
$N_n\equiv5$~{\em (mod~$27$)} if, and only if, 
$n=4\cdot3^i$ with $i\ge1;$

\item[(vi)] 
$N_n\equiv6$~{\em (mod~$27$)} if, and only if, 
$n=22\cdot3^i$ with $i\ge0$,
or
$n=3^{i_1}+3^{i_2}+1$ with $i_1-1>i_2\ge2$,
or
$n=4\cdot3^{i_1}+3^{i_2}$ with $i_1-1>i_2\ge1$,
or
$n=3^{i_1}+4\cdot3^{i_2}$ with $i_1-2>i_2\ge1$,
or
$n=2\cdot3^{i_1}+3^{i_2}+3^{i_3}$ with $i_1-2>i_2-1>i_3\ge0$,
or
$n=3^{i_1}+2\cdot3^{i_2}+3^{i_3}$ with $i_1-1>i_2-1>i_3\ge0$,
or
$n=3^{i_1}+3^{i_2}+2\cdot3^{i_3}$ with $i_1-1>i_2>i_3\ge0;$

\item[(vii)] 
$N_n\equiv9$~{\em (mod~$27$)} if, and only if, 
$n=2\cdot3^{i_1}+2\cdot3^{i_2}+3^{i_3}$ with $i_1>i_2>i_3\ge0$,
or
$n=2\cdot3^{i_1}+3^{i_2}+2\cdot3^{i_3}$ with $i_1>i_2>i_3\ge0$,
or
$n=3^{i_1}+2\cdot3^{i_2}+2\cdot3^{i_3}$ with $i_1>i_2>i_3\ge0$,
or
$n=2\cdot3^{i_1}+2\cdot3^{i_2}+2\cdot3^{i_3}$ with $i_1>i_2>i_3\ge0$,
or
$n=3^{i_1}+3^{i_2}+3^{i_3}+3^{i_4}+3^{i_5}$ 
with $i_1>i_2>i_3>i_4>i_5\ge0$,
or
$n=2\cdot3^{i_1}+3^{i_2}+3^{i_3}+3^{i_4}+3^{i_5}$ 
with $i_1>i_2>i_3>i_4>i_5\ge0$,
or
$n=3^{i_1}+2\cdot3^{i_2}+3^{i_3}+3^{i_4}+3^{i_5}$ 
with $i_1>i_2>i_3>i_4>i_5\ge0$,
or
$n=3^{i_1}+3^{i_2}+2\cdot3^{i_3}+3^{i_4}+3^{i_5}$ 
with $i_1>i_2>i_3>i_4>i_5\ge0$,
or
$n=3^{i_1}+3^{i_2}+3^{i_3}+2\cdot3^{i_4}+3^{i_5}$ 
with $i_1>i_2>i_3>i_4>i_5\ge0$,
or
$n=3^{i_1}+3^{i_2}+3^{i_3}+3^{i_4}+2\cdot3^{i_5}$ 
with $i_1>i_2>i_3>i_4>i_5\ge0;$

\item[(viii)] 
$N_n\equiv11$~{\em (mod~$27$)} if, and only if, 
$n=3^i+1$ with $i\ge3;$

\item[(ix)] 
$N_n\equiv12$~{\em (mod~$27$)} if, and only if, 
$n=7;$
or
$n=40\cdot3^i$ with $i\ge0$,
or
$n=2\cdot3^{i_1}+3^{i_2}+1$ with $i_1-1>i_2\ge1$,
or
$n=3^{i_1}+2\cdot3^{i_2}+1$ with $i_1-1>i_2\ge1$,
or
$n=3^{i_1}+3^{i_2}+3^{i_3}+3^{i_4}$ with $i_1-3>i_2-2>i_3-1>i_4\ge0;$

\item[(x)] 
$N_n\equiv13$~{\em (mod~$27$)} if, and only if, 
$n=3^i$ with $i\ge2;$

\item[(xi)] 
$N_n\equiv15$~{\em (mod~$27$)} if, and only if, 
$n=3^{i}+4$ with $i\ge3$,
or
$n=4\cdot3^{i}+1$ with $i\ge2$,
or
$n=13\cdot3^i$ with $i\ge1$,
or
$n=4\cdot3^{i_1}+2\cdot3^{i_2}$ with $i_1-1>i_2\ge0$,
or
$n=2\cdot3^{i_1}+4\cdot3^{i_2}$ with $i_1-2>i_2\ge0;$

\item[(xii)] 
$N_n\equiv18$~{\em (mod~$27$)} if, and only if, 
$n=2\cdot3^{i_1}+3^{i_2}+3^{i_3}+3^{i_4}$ 
with $i_1>i_2>i_3>i_4\ge0$,
or
$n=3^{i_1}+2\cdot3^{i_2}+3^{i_3}+3^{i_4}$ 
with $i_1>i_2>i_3>i_4\ge0$,
or
$n=3^{i_1}+3^{i_2}+2\cdot3^{i_3}+3^{i_4}$ 
with $i_1>i_2>i_3>i_4\ge0$,
or
$n=3^{i_1}+3^{i_2}+3^{i_3}+2\cdot3^{i_4}$ 
with $i_1>i_2>i_3>i_4\ge0$,
or
$n=2\cdot3^{i_1}+2\cdot3^{i_2}+3^{i_3}+3^{i_4}$ 
with $i_1>i_2>i_3>i_4\ge0$,
or
$n=2\cdot3^{i_1}+3^{i_2}+2\cdot3^{i_3}+3^{i_4}$ 
with $i_1>i_2>i_3>i_4\ge0$,
or
$n=2\cdot3^{i_1}+3^{i_2}+3^{i_3}+2\cdot3^{i_4}$ 
with $i_1>i_2>i_3>i_4\ge0$,
or
$n=3^{i_1}+2\cdot3^{i_2}+2\cdot3^{i_3}+3^{i_4}$ 
with $i_1>i_2>i_3>i_4\ge0$,
or
$n=3^{i_1}+2\cdot3^{i_2}+3^{i_3}+2\cdot3^{i_4}$ 
with $i_1>i_2>i_3>i_4\ge0$,
or
$n=3^{i_1}+3^{i_2}+2\cdot3^{i_3}+2\cdot3^{i_4}$ 
with $i_1>i_2>i_3>i_4\ge0$,
or
$n=3^{i_1}+3^{i_2}+3^{i_3}+3^{i_4}+3^{i_5}+3^{i_6}$ 
with $i_1>i_2>i_3>i_4>i_5>i_6\ge0;$

\item[(xiii)] 
$N_n\equiv20$~{\em (mod~$27$)} if, and only if, 
$n=10$,
or
$n=3^{i_1}+3^{i_2}$ with $i_1-2>i_2\ge1;$

\item[(xiv)] 
$N_n\equiv21$~{\em (mod~$27$)} if, and only if, 
$n=7\cdot 3^i$ with $i\ge1$,
or
$n=4\cdot3^{i_1}+4\cdot3^{i_2}$ with $i_1-2>i_2\ge0$,
or
$n=13\cdot3^{i_1}+3^{i_2}$ with $i_1-1>i_2\ge0$,
or
$n=3^{i_1}+13\cdot3^{i_2}$ with $i_1-3>i_2\ge0;$

\item[(xv)] 
$N_n\equiv23$~{\em (mod~$27$)} if, and only if, 
$n=4;$

\item[(xvi)] 
$N_n\equiv24$~{\em (mod~$27$)} if, and only if, 
$n=13$,
or
$n=16\cdot 3^i$ with $i\ge0$,
or
$n=7\cdot3^{i_1}+3^{i_2}$ with $i_1-1>i_2\ge0;$

\item[(xvii)] in the cases not covered by items {\em(i)}--{\em(xvi),}
$N_n$ is divisible by $27;$
in particular, $N_n \not\equiv 7,8,10,14,16,17,19,22,25,26$~{\em(mod~$27$)} 
for all $n$.
\end{enumerate}
\end{corollary}

\begin{proof}
We convert the right-hand side of \eqref{eq:Loes1} into a linear
combination of series $H_{a_1,a_2,\dots,a_r}(z)$ with all $a_i$'s
relatively prime to~$p$, by means of Equation~\eqref{eq:Phipot} and 
Proposition~\ref{conj:H}. The result is
\begin{multline*} 
\sum _{n\ge1} ^{}N_n\,z^n
= 9 z^2 + 15 z + (18 z + 13) H_{1}(z) 
+ (18 z + 1) H_{2}(z) 
 + (9 z + 20) H_{1, 1}(z)\\
+ 12 H_{2, 1}(z) 
 + 12 H_{1, 2}(z) 
+ 24 H_{1, 1, 1}(z) 
+ 12 H_{4}(z)
+ 3 H_{2, 2}(z) \\
+ 6 H_{2, 1, 1}(z) 
+ 6 H_{1, 2, 1}(z) 
+ 6 H_{1, 1, 2}(z) 
+ 12 H_{1, 1, 1, 1}(z) 
+ 9 H_{4, 1}(z) 
+ 9 H_{1, 4}(z) \\
+ 9 H_{2, 2, 1}(z) 
+ 9 H_{2, 1, 2}(z) 
+ 9 H_{1, 2, 2}(z) 
+ 18 H_{2, 1, 1, 1}(z) 
+ 18 H_{1, 2, 1, 1}(z) \\
+ 18 H_{1, 1, 2, 1}(z) 
+ 18 H_{1, 1, 1, 2}(z) 
+ 9 H_{1, 1, 1, 1, 1}(z) 
+ 9 H_{4, 2}(z) 
+ 9 H_{2, 4}(z) \\
+ 18 H_{4, 1, 1}(z) 
+ 18 H_{1, 4, 1}(z) 
+ 18 H_{1, 1, 4}(z) 
+ 9 H_{2, 2, 2}(z) 
+ 18 H_{2, 2, 1, 1}(z) 
+ 18 H_{2, 1, 2, 1}(z)\\ 
+ 18 H_{2, 1, 1, 2}(z) 
+ 18 H_{1, 2, 2, 1}(z) 
+ 18 H_{1, 2, 1, 2}(z) 
+ 18 H_{1, 1, 2, 2}(z) \\
+ 9 H_{2, 1, 1, 1, 1}(z) 
+ 9 H_{1, 2, 1, 1, 1}(z) 
+ 9 H_{1, 1, 2, 1, 1}(z) \\
+ 9 H_{1, 1, 1, 2, 1}(z) 
+ 9 H_{1, 1, 1, 1, 2}(z) 
+ 18 H_{1, 1, 1, 1, 1, 1}(z)
\quad \quad 
\text {modulo }27.
\end{multline*}
Coefficient extraction following the algorithm described in
Remark~\ref{rem:eff} then yields the claimed congruences.
\end{proof}

Considered modulo~$9$, the above corollary reduces to the following.

\begin{corollary} \label{cor:N9}
The numbers $N_n$ of connected non-crossing  graphs
obey the following congruences modulo $9${\em:}

\begin{enumerate}
\item[(i)] 
$N_n\equiv1$~{\em (mod~$9$)} if, and only if, 
$n=1$,
or
$n=2\cdot3^i$ with $i\ge0$;

\item[(ii)] 
$N_n\equiv2$~{\em (mod~$9$)} if, and only if, 
$n=3^{i_1}+3^{i_2}$ 
with $i_1-1>i_2\ge0;$

\item[(iii)] 
$N_n\equiv3$~{\em (mod~$9$)} if, and only if, 
$n=2\cdot3^{i_1}+3^{i_2}$ 
with $i_1>i_2\ge0$,
or
$n=3^{i_1}+2\cdot3^{i_2}$ 
with $i_1>i_2\ge0$,
or
$n=2\cdot3^{i_1}+2\cdot3^{i_2}$ 
with $i_1>i_2\ge0$,
or
$n=3^{i_1}+3^{i_2}+3^{i_3}+3^{i_4}$ 
with $i_1>i_2>i_3>i_4\ge0;$

\item[(iv)] 
$N_n\equiv4$~{\em (mod~$9$)} if, and only if, 
$n=3^i$ with $i\ge0;$

\item[(v)] 
$N_n\equiv5$~{\em (mod~$9$)} if, and only if, 
$n=4\cdot3^i$ with $i\ge0;$

\item[(vi)] 
$N_n\equiv6$~{\em (mod~$9$)} if, and only if, 
$n=3^{i_1}+3^{i_2}+3^{i_3}$ with $i_1>i_2>i_3\ge0$,
or
$n=2\cdot3^{i_1}+3^{i_2}+3^{i_3}$ with $i_1>i_2>i_3\ge0$,
or
$n=3^{i_1}+2\cdot3^{i_2}+3^{i_3}$ with $i_1>i_2>i_3\ge0$,
or
$n=3^{i_1}+3^{i_2}+2\cdot3^{i_3}$ with $i_1>i_2>i_3\ge0;$

\item[(vii)] in the cases not covered by items {\em(i)}--{\em(vi),}
$N_n$ is divisible by $9;$
in particular, $N_n \not\equiv 7,8$~{\em(mod~$9$)} 
for all $n$.
\end{enumerate}
\end{corollary}

%
%

\begin{remark}
In the first proof of the conjecture \eqref{eq:DeSaAA} in
\cite{EuLYAA}, Eu, Liu and Yeh worked from the sum expression
$$
N_n=\frac {1} {n-1}\sum_{i=n-2}^{2n-4}\binom {3n-3}{n+i+1}\binom {i}{n-2}
$$
which is provided in \cite[Theorem~2]{FN}. In order to achieve the
proof, they considered four auxiliary sums of very similar nature,
called there $f_i(n)$, $i=1,2,3,4$, of which they
determined their behaviour modulo~$3$. These sums, and the numbers
$N_n$ of connected non-crossing graphs, are reconsidered by Gessel
in \cite{GessAZ}, who provides simpler proofs of the congruences
modulo~$3$ established in \cite{EuLYAA}, and also closed form
expressions for the sums $f_i(n)$, $i=1,2,3,4$. Gessel adds
another sum, called $f_5(n)$, of which he also determines its
behaviour modulo~$3$.

All these sums can be treated in the same way as the numbers $N_n$
in Theorems~\ref{thm:N}, \ref{thm:noncr27}, and
Corollaries~\ref{cor:N27}, \ref{cor:N9}, as we are going to explain now. 
In particular, in this way the behaviour of any of $f_i(n)$,
$i=1,2,3,4,5$, modulo {\it any} power of~$3$ can be determined.

The general family
of sums introduced by Gessel in \cite{GessAZ} is
$$
g_{j,k,l}(n)=\sum_{i=n-l}^{2n+j-k}\binom {3n+j}{n+i+k}\binom i{n-l},
$$
of which the sums $f_i(n)$, $i=1,2,3,4,5$, are the cases
$(j,k,l)=(1,1,0)$, $(0,1,0)$, $(0,0,0)$, $(-1,1,1)$, $(0,1,1)$, 
in this order.
Gessel \cite[Sec.~3]{GessAZ} proves that the generating function
$$
G_{j,k,l}(z)=\sum_{n=k-j-l}^\infty g_{j,k,l}(n) z^n
$$
for these sums is given by
$$
G_{j,k,l}(z)=\frac {(1-2\al(z))^l\,\al^{k-j-l}(z)} 
{(1-6\al(z)+6\al^2(z))\,(1-\al(z))^{k-1}},
$$
where $\al(z)$ is the compositional inverse of $z(1-z)(1-2z)$.
The series $\al(z)$ thus being algebraic of degree~$3$, also
all series $G_{j,k,l}(z)$ are algebraic of degree~$3$.
Below we list the algebraic equations satisfied by the generating
functions for $f_i(n)$, $i=1,2,3,4,5$, and in each case we provide
the base solution $F_1(z)$ for the modulus $3^{3^\al}$ 
(cf.\ \eqref{eq:Ansatz1} with $p=3$) 
for starting our method described in Section~\ref{sec:method}.
The iterative steps carry through as well, as in the proof
of Theorem~\ref{thm:N}. We leave the details to the reader.

The generating function for the numbers $f_1(n)$ satisfies
the equation
\begin{equation} \label{eq:f1}
(1-108z^2)G^3_{1,1,0}(z) 
-3G_{1,1,0}(z) +2=0.
\end{equation}
A base solution for the modulus $3^{3^\al}$ is $F_1(z)=1$.
The generating function for the numbers $f_2(n)$ satisfies
the equation
\begin{equation} \label{eq:f2}
(1-108z^2)G^3_{0,1,0}(z) 
-(1+9z)G_{0,1,0}(z) +z=0,
\end{equation}
and a base solution is $F_1(z)=\Phi^{3^\al}(z)+s_\al(z)$.
Next, the generating function for the numbers $f_3(n)$ satisfies
the equation
\begin{equation} \label{eq:f3}
(1-108z^2)G^3_{0,0,0}(z) 
-(1+9z)G_{0,0,0}(z) -z=0.
\end{equation}
Here, a base solution is $F_1(z)=-\Phi^{3^\al}(z)-s_\al(z)+1$.
The generating function for the numbers $f_4(n)$ satisfies
the equation
\begin{equation} \label{eq:f4}
(1-108z^2)G^3_{-1,1,1}(z) 
+(1-108z^2)G^2_{-1,1,1}(z) 
+3z(1-12z)G_{-1,1,1}(z) -4z^2=0.
\end{equation}
Since, modulo~$3$, this is the same equation as \eqref{eq:fmdiff},
this generating function has the same base solution as $N(z)$,
namely the one given by \eqref{eq:baseN}.
Finally, the generating function for the numbers $f_5(n)$ satisfies
the equation
\begin{equation} \label{eq:f5}
(1-108z^2)G^3_{0,1,1}(z) 
-G_{0,1,1}(z) -8z=0.
\end{equation}
Since, modulo~$3$, this is the same equation as \eqref{eq:f2}, 
(essentially) the same base solution (the difference occurring in the
constant term), 
namely $F_1(z)=\Phi^{3^\al}(z)+s_\al(z)+1$, applies.
\end{remark}

\section{The number of Kreweras walks modulo $3$-powers}
\label{sec:Krew}

The subject of this section are congruences modulo powers of~$3$ 
for the numbers of so-called ``Kreweras walks." In his thesis
\cite{KrewAB}, Kreweras had considered (among many other things)
a three-candidate ballot problem, where each of the three candidates,
A, B, C, say,
receives $n$~votes. The problem which is posed is how many ways
there are to count the votes so that at each point in time candidate~A
has at least as many votes as the number of votes for each of the
other candidates. It is not difficult to see that this problem can
be translated into the following problem of counting walks
in the quarter plane: how many
lattice walks in the plane 
from the origin to itself are there which consist of $3n$ steps from the set
$\{(1,1),\, (-1,0),\, (0,-1)\}$ and stay in the
non-negative quadrant?
Let $K_n$ denote the number of these walks. By definition, we set $K_0=1$.
Kreweras \cite{KrewAB} proved that
\begin{equation} \label{eq:Kn} 
K_n=\frac {4^n} {(n+1)(2n+1)}\binom {3n}n;
\end{equation}
see also \cite{NiedAJ}. To prove this formula turned out to be
highly non-trivial, as well as the enumerative analysis of
the more general question of how many walks there are which
end at an arbitrary point in the non-negative quadrant.
Kreweras \cite{KrewAB} and Niederhausen \cite{NiedAJ} obtained
partial results in this direction. Finally, Bousquet-M\'elou
\cite{BousAH} succeeded to compute the complete generating
function for these numbers of walks.

Let $K(z):=\sum _{n\ge0} ^{}K_n\,z^n$ be the generating
function for the ``Kreweras numbers'' in \eqref{eq:Kn}.
From \cite[Theorem~1]{BousAH},
it can be derived that
\begin{equation} \label{eq:Kdiff}
64z^2\,K^3(z) +16z\, K^2(z) - (72z-1)\,K(z) + 54z-1 = 0.
\end{equation}
Clearly, this equation determines $K(z)$
uniquely as a formal power series over~$\Z$, and as well modulo any
power of~$3$.

\begin{theorem} \label{thm:K}
Let $\Phi(z)=\sum _{n\ge0} ^{}z^{3^n},$ and let $\al$ be a
non-negative integer.
Then the generating function $K(z),$ when reduced modulo $3^{3^\al},$ 
can be expressed as a polynomial in $\Phi(z^{1/2})$ of degree at most
$3^{\al+1}-1,$ with coefficients that are Laurent polynomials in $z^{1/2}$
over the integers.
\end{theorem}

\begin{proof} 
We apply the method from Section~\ref{sec:method}. We start by
substituting the Ansatz \eqref{eq:Ansatz1} with $z$ replaced
by $z^{1/2}$ in \eqref{eq:Kdiff} and reducing the result modulo $3$. 
In this way, we obtain
\begin{multline} \label{eq:K-1}
z^2\sum _{i=0} ^{3^{\al+1}-1}a_{i,1}^3(z^{1/2})\Phi^{3i}(z^{1/2})+
z\sum _{i=0} ^{3^{\al+1}-1}a_{i,1}^2(z^{1/2})\Phi^{2i}(z^{1/2})\\
-z\sum _{0\le i<j\le 3^{\al+1}-1}a_{i,1}(z^{1/2})a_{j,1}(z^{1/2})\Phi^{i+j}(z^{1/2})\\
+\sum _{i=0} ^{3^{\al+1}-1}a_{i,1}(z^{1/2})\Phi^{i}(z^{1/2})
-1=0\quad \text
{modulo }3.
\end{multline}
We claim that the following choices solve the above congruence:
\begin{align*}
a_{0,1}(z^{1/2})&=z^{-1}s^2_\al(z^{1/2})
\quad \text {modulo }3,\\[2mm]
a_{3^{\al},1}(z^{1/2})&=-z^{-1}s_\al(z^{1/2})
\quad \text {modulo }3,\\
a_{2\cdot 3^{\al},1}(z^{1/2})&=z^{-1}\quad \text {modulo }3,
\end{align*}
where $s_\al(z)$ is the same polynomial as in
Section~\ref{sec:noncr}, 
with all other $a_{i,1}(z^{1/2})$ vanishing.
In order to verify the claim, we substitute our choices in
\eqref{eq:K-1}. For the left-hand side, we obtain
\begin{multline*} 
z^2\left(z^{-3}\Phi^{2\cdot 3^{\al+1}}(z^{1/2})
-z^{-3}s^3_\al(z^{1/2})\Phi^{3^{\al+1}}(z^{1/2})+z^{-3}s^6_\al(z^{1/2})
\right)\\
+z\left(z^{-2}\Phi^{4\cdot 3^{\al}}(z^{1/2})
+z^{-2}s^2_\al(z^{1/2})\Phi^{2\cdot 3^{\al}}(z^{1/2})+z^{-2}s^4_\al(z^{1/2})
\right)\\
-z\left(
-z^{-2}s_\al(z^{1/2})\Phi^{3^{\al+1}}(z^{1/2})
+z^{-2}s^2_\al(z^{1/2})\Phi^{2\cdot 3^{\al}}(z^{1/2})
-z^{-2}s^3_\al(z^{1/2})\Phi^{3^{\al}}(z^{1/2})
\right)\\
+z^{-1}\Phi^{2\cdot 3^{\al}}(z^{1/2})
-z^{-1}s_\al(z^{1/2})\Phi^{3^{\al}}(z^{1/2})+z^{-1}s^2_\al(z^{1/2})
-1
\quad \quad \text{modulo 3}.
\end{multline*}
Using the relation \eqref{eq:PhiRel2} with $p=3$ and $z$ replaced
by $z^{1/2}$,
and reducing the obtained expression modulo~3,
we arrive at
\begin{multline*} 
\left(z^{-1}
\Phi^{2\cdot 3^{\al}}(z^{1/2})
+\Phi^{3^{\al}}(z^{1/2})z^{3^{\al}/2-1}+z^{3^{\al}-1}\right)\\
-\left(z^{-1}s^3_\al(z^{1/2})
\Phi^{3^{\al}}(z^{1/2})-z^{3^{\al}/2-1}s^3_\al(z^{1/2})\right)
+z^{-1}s^6_\al(z^{1/2})
\\
+\left(z^{-1}
\Phi^{2\cdot 3^{\al}}(z^{1/2})-z^{3^{\al}/2-1}\Phi^{3^{\al}}(z^{1/2})\right)
+z^{-1}s^2_\al(z^{1/2})\Phi^{2\cdot 3^{\al}}(z^{1/2})+z^{-1}s^4_\al(z^{1/2})
\\
+\left(z^{-1}s_\al(z^{1/2})
\Phi^{3^{\al}}(z^{1/2})-z^{3^{\al}/2-1}s_\al(z^{1/2})\right)\\
-z^{-1}s^2_\al(z^{1/2})\Phi^{2\cdot 3^{\al}}(z^{1/2})
+z^{-1}s^3_\al(z^{1/2})\Phi^{3^{\al}}(z^{1/2})
\\
+z^{-1}\Phi^{2\cdot 3^{\al}}(z^{1/2})
-z^{-1}s_\al(z^{1/2})\Phi^{3^{\al}}(z^{1/2})+z^{-1}s^2_\al(z^{1/2})
-1
\quad \quad \text{modulo 3}.
\end{multline*}
By collecting terms, this expression simplifies to 
\begin{multline*}
z^{3^{\al}-1}
+z^{3^{\al}/2-1}s^3_\al(z^{1/2})
+z^{-1}s^6_\al(z^{1/2})
+z^{-1}s^4_\al(z^{1/2})\\
-z^{3^{\al}/2-1}s_\al(z^{1/2})
+z^{-1}s^2_\al(z^{1/2})
-1
\quad \quad \text{modulo 3}.
\end{multline*}
By repeatedly using the relation \eqref{eq:sal} with $z$ replaced by
$z^{1/2}$, this expression can be turned into
\begin{multline*}
z^{3^{\al}-1}
+z^{3^{\al}/2-1}\left(s_\al(z^{1/2})+z^{3^{\al}/2}-z^{1/2}\right)
+z^{-1}\left(s_\al(z^{1/2})+z^{3^{\al}/2}-z^{1/2}\right)^2\\
+z^{-1}s_\al(z^{1/2})\left(s_\al(z^{1/2})+z^{3^{\al}/2}-z^{1/2}\right)
-z^{3^{\al}/2-1}s_\al(z^{1/2})
+z^{-1}s^2_\al(z^{1/2})
-1
\quad \quad \text{modulo 3}.
\end{multline*}
After expansion and reduction modulo~$3$, one sees that this
expression reduces to zero.

\medskip
After we have completed the ``base  step," we now proceed with the
iterative steps described in Section~\ref{sec:method}. We consider
the Ansatz \eqref{eq:Ansatz2}--\eqref{eq:Ansatz2b} with $z$ replaced
by~$z^{1/2}$, where the
coefficients $a_{i,\be}(z^{1/2})$ are supposed to provide a solution
$$F_{\be}(z)=\sum _{i=0}
^{3^{\al+1}-1}a_{i,\be}(z^{1/2})\Phi^i(z^{1/2})$$ 
to
\eqref{eq:Kdiff} modulo~$3^\be$. This Ansatz, substituted in
\eqref{eq:Kdiff}, produces the congruence
\begin{multline} \label{eq:ItGlK}
-3^\be z\sum _{j=0} ^{3^{\al+1}-1}\sum _{i=0} ^{3^{\al+1}-1}
a_{j,\be}(z^{1/2})b_{i,\be+1}(z^{1/2})\Phi^{i+j}(z^{1/2})
+3^\be\sum _{i=0} ^{3^{\al+1}-1}b_{i,\be+1}(z^{1/2})\Phi^{i}(z^{1/2})\\
+64z^2F^3_{\be}(z)
+16z F^2_{\be}(z)-(72z-1)F_{\be}(z)+54z-1=0
\quad \quad 
\text {modulo }3^{\be+1}.
\end{multline}
Since the sum has the prefactor $3^\be$, we may reduce the
$a_{i,\be}(z^{1/2})$ modulo $3$. By construction, we have
\begin{align*}
a_{0,\be}(z^{1/2})&=a_{0,1}(z^{1/2})=z^{-1}s^2_\al(z^{1/2})
\quad \quad \text{modulo }3,\\
a_{3^\al,\be}(z^{1/2})&=a_{3^\al,1}(z^{1/2})=-z^{-1}s_\al(z^{1/2})
\quad \quad \text{modulo }3,\\
a_{2\cdot3^\al,\be}(z^{1/2})&=a_{2\cdot3^\al,1}(z^{1/2})=z^{-1}
\quad \quad \text{modulo }3,
\end{align*}
and $a_{i,\be}(z^{1/2})=0$ modulo 3 for all other $i$'s.
If we substitute this in \eqref{eq:ItGlK} and subsequently use
\eqref{eq:PhiRel2} with $p=3$ and $z$ replaced by $z^{1/2}$ 
to reduce high powers of $\Phi(z^{1/2})$, we obtain
\begin{multline*} 
-3^\be \sum _{i=0} ^{3^{\al+1}-1}
s^2_\al(z^{1/2})b_{i,\be+1}(z^{1/2})\Phi^{i}(z^{1/2})
+3^\be \sum _{i=0} ^{3^{\al+1}-1}
s_\al(z^{1/2})b_{i,\be+1}(z^{1/2})\Phi^{i+3^\al}(z^{1/2})\\
-3^\be \sum _{i=0} ^{3^{\al+1}-1}
b_{i,\be+1}(z^{1/2})\Phi^{i+2\cdot 3^\al}(z^{1/2})
+3^\be\sum _{i=0} ^{3^{\al+1}-1}b_{i,\be+1}(z^{1/2})\Phi^{i}(z^{1/2})\\
+64z^2F^3_{\be}(z)
+16z F^2_{\be}(z)-(72z-1)F_{\be}(z)+54z-1\\
=
-3^\be
\sum _{i=0} ^{3^{\al}-1}
\left(
s^2_\al(z^{1/2})b_{i,\be+1}(z^{1/2})
+z^{3^\al/2}s_\al(z^{1/2})b_{i+2\cdot 3^\al,\be+1}(z^{1/2})
\right.
\kern3cm
\\
\kern3cm
\left.
-z^{3^\al/2}b_{i+3^\al,\be+1}(z^{1/2})
-b_{i,\be+1}(z^{1/2})
\right)
\Phi^{i}(z^{1/2})\\
-3^\be
\sum _{i=3^\al} ^{2^\cdot3^{\al}-1}
\left(
s^2_\al(z^{1/2})b_{i,\be+1}(z^{1/2})
-s_\al(z^{1/2})b_{i-3^\al,\be+1}(z^{1/2})
\right.
\kern3cm
\\
\kern3cm
\left.
-s_\al(z^{1/2})b_{i+3^\al,\be+1}(z^{1/2})
-z^{3^\al/2}b_{i+3^\al,\be+1}(z^{1/2})
\right)
\Phi^{i}(z^{1/2})\\
-3^\be
\sum _{i=2\cdot3^\al} ^{3^{\al+1}-1}
\left(
s^2_\al(z^{1/2})b_{i,\be+1}(z^{1/2})
-s_\al(z^{1/2})b_{i-3^\al,\be+1}(z^{1/2})
+b_{i-2\cdot 3^\al,\be+1}(z^{1/2})
\right)
\Phi^{i}(z^{1/2})\\
+64z^2F^3_{\be}(z)
+16z F^2_{\be}(z)-(72z-1)F_{\be}(z)+54z-1=0
\quad \quad 
\text {modulo }3^{\be+1}.
\end{multline*}
By our assumption on $F_{\be}(z)$, we may divide by $3^\be$.
Comparison of powers of $\Phi(z^{1/2})$ then yields a system of congruences
of the form
\begin{equation} \label{eq:AbcK} 
M\cdot b=c\quad \text{modulo }3,
\end{equation}
where $b$ is the column vector of unknowns
$(b_{i,\be+1}(z^{1/2}))_{i=0,1,\dots,3^{\al+1}-1}$, $c$ is a (known)
column vector of Laurent polynomials in $z^{1/2}$, and $M$ is the matrix
$$
\begin{pmatrix} 
D\left(s^2_\al(z^{1/2})-1\right)&D\left(-z^{3^\al/2}\right)&
D\left(z^{3^\al/2}s_\al(z^{1/2})\right)\\
D\left(-s_\al(z^{1/2})\right)&D\left(s^2_\al(z^{1/2})\right)&
D\left(-s_\al(z^{1/2})-z^{3^\al/2}\right)\\
D(1)&D\left(-s_\al(z^{1/2})\right)&D\left(s^2_\al(z^{1/2})\right)
\end{pmatrix},
$$
with $D(x)$ denoting the $3^\al\times3^\al$ diagonal matrix whose
diagonal entries equal $x$, as before. 
In the same manner as in the proof of Theorem~\ref{thm:N},
on sees that
$$
\det (M)=z^{3^{\al}}\quad \quad \text{modulo }3.
$$
As a consequence,
the system \eqref{eq:AbcK} is
(uniquely) solvable. Thus, we have proved that, for an arbitrary non-negative
integer $\al$, the algorithm of
Section~\ref{sec:method} will produce a solution $F_{{3^{\al}}}(z)$ 
to \eqref{eq:Kdiff} modulo $3^{3^\al}$ which is a
polynomial in $\Phi(z^{1/2})$ with coefficients that are Laurent polynomials in
$z^{1/2}$.
\end{proof}

We have implemented this algorithm. As an illustration, the next
theorem contains the result for the modulus $27$.

{
\allowdisplaybreaks
\begin{theorem} \label{thm:K27}
Let $\Phi(z)=\sum _{n\ge0} ^{}z^{3^n}$.
Then we have
\begin{multline} 
\label{eq:LoesK1}
\sum _{n\ge0} ^{}K_n\,z^n
=
1
+ 12 \Phi^2(z^{1/2})
+ 11 z^{-1/2} \Phi^3(z^{1/2}) 
+ 6 z^{-1} \Phi^4(z^{1/2}) 
+ 15 z^{-1/2} \Phi^5(z^{1/2}) \\
+ z^{-1} \Phi^6(z^{1/2}) 
+ 3 z^{-1} \Phi^8(z^{1/2})
\quad \quad 
\text {\em modulo }27.
\end{multline}
\end{theorem}
}

\begin{corollary} \label{cor:K}
The numbers $K_n$ of Kreweras walks
obey the following congruences modulo $27${\em:}

\begin{enumerate}
\item[(i)] 
$K_n\equiv1$~{\em (mod~$27$)} if, and only if, $n=0;$

\item[(ii)] 
$K_n\equiv2$~{\em (mod~$27$)} if, and only if, $n=1;$

\item[(iii)] 
$K_n\equiv3$~{\em (mod~$27$)} if, and only if, 
$n=3^i$ with $i\ge1$,
or
$n=\frac {1} {2}(13\cdot 3^i-1)$ with $i\ge2$,
or
$n=\frac {1} {2}(3^i+11)$ with $i\ge4$,
or
$n=2\cdot 3^i+1$ with $i\ge3$,
or
$n=20\cdot 3^i-1$ with $i\ge1$,
or
$n=\frac {1} {2}(3^{i_1}+3^{i_2}+3^{i_3}+3^{i_4}-2)$ 
with $i_1-3>i_2-2>i_3-1>i_4\ge1;$

\item[(iv)] 
$K_n\equiv5$~{\em (mod~$27$)} if, and only if, 
$n=\frac {1} {2}(3^{i_1}+3^{i_2}-2)$ with $i_1-2>i_2\ge2;$

\item[(v)] 
$K_n\equiv6$~{\em (mod~$27$)} if, and only if, 
$n=3^{i}+1$ with $i\ge3$,
or
$n=\frac {1} {2}(3^{i_1}+3^{i_2})$ with $i_1-2>i_2\ge1$,
or
$n=\frac {1} {2}(7\cdot3^{i_1}+3^{i_2}-2)$ with $i_1-1>i_2\ge1;$
or
$n=\frac {1} {2}(3^{i_1}+7\cdot3^{i_2}-2)$ with $i_1-1>i_2\ge1;$

\item[(vi)] 
$K_n\equiv7$~{\em (mod~$27$)} if, and only if, 
$n=3^i-1$ with $i\ge2;$

\item[(vii)] 
$K_n\equiv8$~{\em (mod~$27$)} if, and only if, 
$n=4$, or $n=2\cdot3^i-1$ with $i\ge2;$

\item[(viii)] 
$K_n\equiv9$~{\em (mod~$27$)} if, and only if, 
$n=3^{i_1}+3^{i_2}+3^{i_3}-1$ with $i_1>i_2>i_3\ge1$,
or
$n=\frac {1} {2}(2\cdot3^{i_1}+3^{i_2}+3^{i_3}+3^{i_4}+3^{i_5}-2)$ 
with $i_1>i_2>i_3>i_4>i_5\ge1$,
or
$n=\frac {1} {2}(3^{i_1}+2\cdot3^{i_2}+3^{i_3}+3^{i_4}+3^{i_5}-2)$ 
with $i_1>i_2>i_3>i_4>i_5\ge1$,
or
$n=\frac {1} {2}(3^{i_1}+3^{i_2}+2\cdot3^{i_3}+3^{i_4}+3^{i_5}-2)$ 
with $i_1>i_2>i_3>i_4>i_5\ge1$,
or
$n=\frac {1} {2}(3^{i_1}+3^{i_2}+3^{i_3}+2\cdot3^{i_4}+3^{i_5}-2)$ 
with $i_1>i_2>i_3>i_4>i_5\ge1$,
or
$n=\frac {1} {2}(3^{i_1}+3^{i_2}+3^{i_3}+3^{i_4}+2\cdot3^{i_5}-2)$ 
with $i_1>i_2>i_3>i_4>i_5\ge1;$

\item[(ix)] 
$K_n\equiv12$~{\em (mod~$27$)} if, and only if, 
$n=\frac {1} {2}(4\cdot3^{i_1}+3^{i_2}-1)$ with $i_1-1>i_2\ge2$,
or
$n=\frac {1} {2}(3^{i_1}+4\cdot3^{i_2}-1)$ with $i_1-2>i_2\ge2$,
or
$n=\frac {1} {2}(3^{i_1}+3^{i_2}+2)$ with $i_1-2>i_2\ge3$,
or
$n=\frac {1} {2}(13\cdot3^{i_1}+3^{i_2}-2)$ with $i_1-1>i_2\ge1$,
or
$n=2\cdot3^{i_1}+2\cdot3^{i_2}-1$ with $i_1-2>i_2\ge1$,
or
$n=\frac {1} {2}(3^{i_1}+13\cdot3^{i_2}-2)$ with $i_1-3>i_2\ge1;$

\item[(x)] 
$K_n\equiv14$~{\em (mod~$27$)} if, and only if, 
$n=\frac {1} {2}(3^i+1)$ with $i\ge4$, or $n=5\cdot 3^i-1$ with $i\ge2;$

\item[(xi)] 
$K_n\equiv15$~{\em (mod~$27$)} if, and only if, 
$n=\frac {1} {2}(7\cdot 3^{i}-1)$ with $i\ge2$,
or
$n=\frac {1} {2}(3^{i}+5)$ with $i\ge2$,
or
$n=2\cdot 3^i$ with $i\ge1$,
or
$n=11\cdot 3^i-1$ with $i\ge1$,
or
$n=\frac {1} {2}(2\cdot3^{i_1}+3^{i_2}+3^ {i_3}-2)$ with $i_1-2>i_2-1>i_3\ge1$,
or
$n=\frac {1} {2}(3^{i_1}+2\cdot3^{i_2}+3^ {i_3}-2)$ with $i_1-1>i_2-1>i_3\ge1$,
or
$n=\frac {1} {2}(3^{i_1}+3^{i_2}+2\cdot3^ {i_3}-2)$ with $i_1-1>i_2>i_3\ge1;$

\item[(xii)] 
$K_n\equiv16$~{\em (mod~$27$)} if, and only if, $n=2;$

\item[(xiii)] 
$K_n\equiv17$~{\em (mod~$27$)} if, and only if, $n=5;$

\item[(xiv)] 
$K_n\equiv18$~{\em (mod~$27$)} if, and only if, 
$n=\frac {1} {2}(2\cdot3^{i_1}+2\cdot3^{i_2}+3^{i_3}+3^{i_4}-2)$ 
with $i_1>i_2>i_3>i_4\ge1$,
or
$n=\frac {1} {2}(2\cdot3^{i_1}+3^{i_2}+2\cdot3^{i_3}+3^{i_4}-2)$ 
with $i_1>i_2>i_3>i_4\ge1$,
or
$n=\frac {1} {2}(2\cdot3^{i_1}+3^{i_2}+3^{i_3}+2\cdot3^{i_4}-2)$ 
with $i_1>i_2>i_3>i_4\ge1$,
or
$n=\frac {1} {2}(3^{i_1}+2\cdot3^{i_2}+2\cdot3^{i_3}+3^{i_4}-2)$ 
with $i_1>i_2>i_3>i_4\ge1$,
or
$n=\frac {1} {2}(3^{i_1}+2\cdot3^{i_2}+3^{i_3}+2\cdot3^{i_4}-2)$ 
with $i_1>i_2>i_3>i_4\ge1$,
or
$n=\frac {1} {2}(3^{i_1}+3^{i_2}+2\cdot3^{i_3}+2\cdot3^{i_4}-2)$ 
with $i_1>i_2>i_3>i_4\ge1$,
or
$n=\frac {1} {2}(3^{i_1}+3^{i_2}+3^{i_3}+3^{i_4}+3^{i_5}+3^{i_6}-2)$ 
with $i_1>i_2>i_3>i_4>i_5>i_6\ge1;$

\item[(xv)] 
$K_n\equiv21$~{\em (mod~$27$)} if, and only if, 
$n=19$, 
or
$n=3^{i_1}+3^{i_2}-1$ with $i_1>i_2\ge1$,
or
$n=\frac {1} {2}(3^{i_1}+3^{i_2}+3^{i_3}-1)$ with $i_1-2>i_2-1>i_3\ge2$,
or
$n=\frac {1} {2}(4\cdot3^{i_1}+3^{i_2}+3^{i_3}-2)$ with $i_1-2>i_2-1>i_3\ge1$,
or
$n=\frac {1} {2}(3^{i_1}+4\cdot3^{i_2}+3^{i_3}-2)$ with $i_1-3>i_2-1>i_3\ge1$,
or
$n=\frac {1} {2}(3^{i_1}+3^{i_2}+4\cdot3^{i_3}-2)$ with $i_1-3>i_2-2>i_3\ge1;$

\item[(xvi)] 
$K_n\equiv23$~{\em (mod~$27$)} if, and only if, $n=14;$

\item[(xvii)] 
$K_n\equiv24$~{\em (mod~$27$)} if, and only if, 
$n=10$, or
$n=\frac {1} {2}(2\cdot 3^{i_1}+3^{i_2}-1)$ with $i_1-1>i_2\ge2$,
or
$n=\frac {1} {2}(3^{i_1}+2\cdot 3^{i_2}-1)$ with $i_1>i_2\ge2$,
or
$n=3^{i_1}+2\cdot3^{i_2}-1$ with $i_1-2>i_2\ge1$,
or
$n=2\cdot3^{i_1}+3^{i_2}-1$ with $i_1>i_2\ge1;$

\item[(xviii)] 
$K_n\equiv26$~{\em (mod~$27$)} if, and only if, 
$n=\frac {1} {2}(3^i-1)$ with $i\ge3;$
\item[(xix)] in the cases not covered by items {\em(i)}--{\em(xviii),}
$K_n$ is divisible by $27;$
in particular, $K_n \not\equiv 4,10,11,13,19,20,22,25$~{\em(mod~$27$)} 
for all $n$.
\end{enumerate}

\end{corollary}

\begin{proof}
By means of Equation~\eqref{eq:Phipot} and 
Proposition~\ref{conj:H},
we convert the right-hand side of \eqref{eq:LoesK1} into a linear
combination of series $H_{a_1,a_2,\dots,a_r}\!\left(z^{1/2}\right)$ 
with all $a_i$'s
relatively prime to~$p$. The result is
\begin{multline*} \label{}
\sum _{n\ge0} ^{}K_n\,z^n
=
21 z^{-1/2}H_{1}(z^{1/2}) 
+ 7z^{-1} H_{2}(z^{1/2}) 
+ 5z^{-1} H_{1, 1}(z^{1/2}) \\
+ 9z^{-1/2} H_{2, 1}(z^{1/2}) 
+ 9z^{-1/2} H_{1, 2}(z^{1/2}) 
+ 18z^{-1/2} H_{1, 1, 1}(z^{1/2}) \\
+ 3z^{-1} H_{4}(z^{1/2}) 
+ 21z^{-1} H_{2, 2}(z^{1/2}) 
+ 15z^{-1} H_{2, 1, 1}(z^{1/2})\\ 
+ 15z^{-1} H_{1, 2, 1}(z^{1/2}) 
+ 15z^{-1} H_{1, 1, 2}(z^{1/2}) 
+ 3z^{-1} H_{1, 1, 1, 1}(z^{1/2})\\ 
+ 9z^{-1} H_{4, 2}(z^{1/2}) 
+ 9z^{-1} H_{2, 4}(z^{1/2}) 
+ 18z^{-1} H_{4, 1, 1}(z^{1/2}) 
+ 18z^{-1} H_{1, 4, 1}(z^{1/2}) \\
+ 18z^{-1} H_{1, 1, 4}(z^{1/2}) 
+ 9z^{-1} H_{2, 2, 2}(z^{1/2}) 
+ 18z^{-1} H_{2, 2, 1, 1}(z^{1/2})\\ 
+ 18z^{-1} H_{2, 1, 2, 1}(z^{1/2}) 
+ 18z^{-1} H_{2, 1, 1, 2}(z^{1/2}) 
+ 18z^{-1} H_{1, 2, 2, 1}(z^{1/2}) \\
+ 18z^{-1} H_{1, 2, 1, 2}(z^{1/2}) 
+ 18z^{-1} H_{1, 1, 2, 2}(z^{1/2}) 
+ 9z^{-1} H_{2, 1, 1, 1, 1}(z^{1/2}) \\
+ 9z^{-1} H_{1, 2, 1, 1, 1}(z^{1/2}) 
+ 9z^{-1} H_{1, 1, 2, 1, 1}(z^{1/2}) 
+ 9z^{-1} H_{1, 1, 1, 2, 1}(z^{1/2}) \\
+ 9z^{-1} H_{1, 1, 1, 1, 2}(z^{1/2}) 
+ 18z^{-1} H_{1, 1, 1, 1, 1, 1}(z^{1/2})
\quad \quad 
\text {modulo }27.
\end{multline*}
Coefficient extraction following the algorithm described in
Remark~\ref{rem:eff} then yields the claimed congruences.
\end{proof}

If we restrict Corollary~\ref{cor:K} to modulus~$3$, then it reduces
to the following simple assertion.

\begin{corollary} \label{cor:K3}
The number $K_n$ of Kreweras walks of length $3n$ is congruent
to $1$ modulo~$3$ if, and only if, $n=3^i-1$ with $i\ge1$,
it is congruent 
to $2$ modulo~$3$ if, and only if, $n=\frac {1}
{2}(3^{i_1}+3^{i_2}-2)$ with $i_1>i_2\ge0$, and it is divisible
by~$3$ in all other cases.
\end{corollary}

\section{Fu\ss--Catalan numbers modulo $p$-powers}
\label{sec:FCat}

In this section, given a positive integer~$h$ and a prime number~$p$, 
we determine the behaviour of the
{\it Fu\ss--Catalan numbers}
\begin{equation} \label{eq:FCat} 
F(n;k):=\frac {1} {n}\binom {kn}{n-1}
\end{equation}
modulo powers of~$p$, provided $k$ is itself a power 
of~$p$, say, $k = p^h$.\footnote{\label{F:Lucas}%
In principle, one could use the
generalisations of Lucas' theorem due to Davis and Webb \cite{DaWeAA}, and to
Granville \cite{GranAA}, respectively, to analyse the explicit
expression for the Fu\ss--Catalan numbers 
modulo a given $p$-power. However, this approach would be
rather cumbersome in comparison with our method, and it is doubtful
that one would be able to derive a result on the same level of 
generality as Theorems~\ref{thm:FCat}, \ref{thm:FCat27}, or
Corollaries~\ref{cor:FCat27}, \ref{cor:F25}.}
These numbers have
numerous combinatorial interpretations; cf.\ \cite[pp.~59--60]{ArmDAA}.

By using the Lagrange inversion formula (see
\cite[Theorem~5.4.2]{StanBI}), it is easy to see that
the generating function $f_{p;h}(z)=1+\sum_{n\ge1}\frac {1} {n}\binom
{p^hn}{n-1}z^n$ satisfies the functional equation
\begin{equation} \label{eq:FCatEq}
zf_{p;h}^{p^h}(z)-f_{p;h}(z)+1=0.
\end{equation}
It is straightforward to verify that this
equation has a unique formal power series solution over $\Z$,
and as well over any power of~$p$. 

In order to determine the coefficients of $f_{p;h}(z)$ 
modulo powers of $p$, we have to use a variant of the series
$\Phi(z)$ defined in \eqref{eq:Phidef}, namely
\begin{equation} \label{eq:Phiphdef}
\Phi_{p;h}(z)=\sum _{n\ge0} ^{}z^{p^{nh}/(p^h-1)}.
\end{equation}
The theorem below generalises Theorem~33 in \cite{KaKMAA} to arbitrary
prime numbers.

\begin{theorem} \label{thm:FCat}
For a prime number $p$ and a positive integer $h,$ 
let $\Phi_{p;h}(z)$ be the series defined in \eqref{eq:Phiphdef}$,$ and
let $\al$ be a further positive integer.
Then the generating function $f_{p;h}(z),$ when reduced modulo $p^{p^\al},$ 
can be expressed as a polynomial in $\Phi_{p;h}(z)$ 
of degree at most
$p^{(\al+1)h}-1,$ with coefficients that are Laurent polynomials in 
$z^{1/(p^h-1)}$ over the integers.
\end{theorem}

\begin{proof} 
For ease of notation, we replace $z$ by $z^{p^h-1}$ in
\eqref{eq:FCatEq}, thereby obtaining the equation
\begin{equation} \label{eq:pmeqA}
z^{p^h-1}\tilde f_{p;h}^{p^h}(z)-\tilde f_{p;h}(z)+1=0,
\end{equation}
with $\tilde f_{p;h}(z)=f_{p;h}(z^{p^h-1})$. We now have
to prove that, modulo $p^{p^{\al h}}$, the series $\tilde
f_{p;h}(z)$ can be expressed as a polynomial in
\begin{equation} \label{eq:Phimdef} 
\tilde \Phi_{p;h}(z)=\sum _{n=0} ^{\infty}z^{p^{nh}}
\end{equation}
of degree at most $p^{(\al+1)h}-1$, with coefficients that are Laurent 
polynomials in $z$.

It is readily verified that
\begin{equation} \label{eq:Phim2}
\tilde\Phi_{p;h}^{p^h}(z) -\tilde\Phi_{p;h}(z)+z=0\quad \text{modulo }p,
\end{equation}
whence
\begin{equation} \label{eq:Phim2al}
\left(\tilde\Phi_{p;h}^{p^h}(z) -\tilde\Phi_{p;h}(z)+z\right)^{p^{\al h}}=0\quad 
\text{modulo }p^{p^{\al h}}.
\end{equation}
We modify our Ansatz \eqref{eq:Ansatz} to
\begin{equation} \label{eq:Ansatzm}
\tilde f_{p;h}(z)=\sum _{i=0} ^{p^{(\al+1)h}-1}a_i(z)\tilde\Phi_{p;h}^i(z)\quad \text 
{modulo
}p^{p^{\al h}},
\end{equation}
where the $a_i(z)$'s
are (at this point) undetermined Laurent polynomials in $z$.

Next, we gradually find approximations $a_{i,\be}(z)$ to $a_i(z)$ such that
\eqref{eq:pmeqA} holds modulo $p^\be$, for $\be=1,2,\dots,
p^{\al h}$. To start the procedure, we consider the functional equation
\eqref{eq:pmeqA} modulo $p$, with
\begin{equation} \label{eq:AnsatzA1}
\tilde f_{p;h}(z)=\sum _{i=0} ^{p^{(\al+1)h}-1}a_{i,1}(z)\tilde\Phi_{p;h}^i(z)\quad \text {modulo
}p.
\end{equation}
It is readily verified that the choice of
\begin{align} \notag
a_{0,1}(z)&=
\sum _{k=0} ^{{\al}-1}z^{p^{kh}-1},\\[2mm]
\label{eq:Ansatza1}
a_{p^{\al h},1}(z)&=z^{-1},
\end{align}
with all other $a_{i,1}(z)$'s equal to zero indeed leads to a
solution of \eqref{eq:pmeqA} modulo~$p$.
 
\medskip
After we have completed the ``base  step," we now proceed with the
iterative steps described in Section~\ref{sec:method}. Our Ansatz
here (replacing the corresponding one in
\eqref{eq:Ansatz2}--\eqref{eq:Ansatz2b}) is
\begin{equation} \label{eq:AnsatzA2}
\tilde f_{p;h}(z)
=\sum _{i=0} ^{p^{(\al+1)h}-1}a_{i,\be+1}(z)\tilde\Phi_{p;h}^i(z)\quad 
\text{modulo }p^{\be+1},
\end{equation}
with
\begin{equation} \label{eq:AnsatzA2a}
a_{i,\be+1}(z):=a_{i,\be}(z)+p^{\be}b_{i,\be+1}(z),\quad 
i=0,1,\dots,p^{(\al+1)h}-1,
\end{equation}
where the
coefficients $a_{i,\be}(z)$ are supposed to provide a solution
$$\tilde f_{\be}(z)=\sum _{i=0}
^{p^{(\al+1)h}-1}a_{i,\be}(z)\tilde\Phi_{p;h}^i(z)$$ 
to
\eqref{eq:pmeqA} modulo~$p^\be$. This Ansatz, substituted in
\eqref{eq:pmeqA}, produces the congruence
\begin{equation} \label{eq:iter}
z^{p^h-1}\tilde f_{\be}^{p^h}(z)-\tilde f_{\be}(z)
+p^\be\sum _{i=0} ^{p^{(\al+1)h}-1}b_{i,\be+1}(z)\tilde\Phi_{p;h}^i(z)
+1=0
\quad 
\text {modulo }p^{\be+1}.
\end{equation}
By our assumption on $\tilde f_{\be}(z)$, we may divide by $p^\be$.
Comparison of powers of $\tilde\Phi_{p;h}(z)$ then yields a system of congruences
of the form
\begin{equation} \label{eq:bi+1} 
b_{i,\be+1}(z)+\text {Pol}_i(z)=0\quad 
\text {modulo }p,\quad \quad 
i=0,1,\dots,p^{(\al+1)h}-1,
\end{equation}
where $\text {Pol}_i(z)$, $i=0,1,\dots,p^{(\al+1)h}-1$, are certain
Laurent polynomials with integer coefficients. This system being trivially
uniquely solvable, we have proved that, for an arbitrary positive
integer $\al$, the modified algorithm that we have presented here
will produce a solution $\tilde f_{{
p^{\al h}}}(z)$ to \eqref{eq:pmeqA} modulo $p^{p^{\al h}}$ which is a
polynomial in $\tilde\Phi_{p;h}(z)$ with coefficients that are Laurent
polynomials in $z$.
\end{proof}

It should be observed that the $\al=0$ case of the above proof
(see in particular \eqref{eq:Ansatza1})
shows that $F(n;p^h)\equiv1$~(mod~$p$) for $n=(p^{hi}-1)/(p-1)$,
$i=0,1,\dots$, and $F(n;p^h)\equiv0$~(mod~$p$) otherwise.

We have implemented the algorithm contained in the above proof. 
As an illustration, we display
below the result obtained for $h=1$ and the modulus $p^2$. This result
was first guessed from the automatically obtained results for
$p=3,5,7$, but, once found, it is easily verified directly by substitution
in \eqref{eq:FCatEq}.

\begin{theorem} \label{thm:FCat27}
Let $\Phi(z)=\sum _{n\ge0} ^{}z^{p^n}$.
Then we have
\begin{multline} 
\label{eq:Loes1FCat}
\sum _{n\ge1} ^{}\frac {1} {n}\binom {pn}{n-1}\,z^n
=
p\Phi^{p-1}\!\left(z^{1/(p-1)}\right)
-(p-1)z^{-1/(p-1)}\Phi^p\!\left(z^{1/(p-1)}\right)\\
+pz^{-1/(p-1)}\Phi^{2p-1}\!\left(z^{1/(p-1)}\right)
\quad \quad 
\text {\em modulo }p^2.
\end{multline}
\end{theorem}

(Human) inspection reveals that this expression can actually be
drastically simplified.

{
\begin{corollary} \label{cor:FCat27}
Let $\Phi(z)=\sum _{n\ge0} ^{}z^{p^n}$.
Then we have
\begin{equation} 
\label{eq:Loes1FCata}
\sum _{n\ge1} ^{}\frac {1} {n}\binom {pn}{n-1}\,z^n
=
z^{-1/(p-1)}\Phi^p\!\left(z^{1/(p-1)}\right)
\quad \quad 
\text {\em modulo }p^2.
\end{equation}
\end{corollary}

\begin{proof}
One applies relation \eqref{eq:PhiRel} with $\al=0$ and $z$ replaced
by $z^{1/(p-1)}$ to the terms in \eqref{eq:Loes1FCat} which have
coefficient~$p$.
\end{proof}

Explicitly, this leads to the following congruences modulo~$p^2$.

\begin{corollary} \label{cor:F25}
The Fu\ss--Catalan numbers $F(n;p)=\frac {1} {n}\binom {pn}{n-1}$
obey the following congruences modulo $p^2${\em:}

\begin{enumerate}
\item[(i)]
If $n=\frac {1} {p-1}\left(p^i-1\right)$ with $i\ge1$,
then $F(n;p)\equiv1$~{\em(mod~$p^2$)}.
\item[(ii)]
If 
$$n=\frac {1} {p-1}\left(a_1p^{i_1}+a_2p^{i_2}+\dots+
a_rp^{i_r}-1\right)$$
with all $a_i$'s relatively prime to~$p$,
$a_1+a_2+\dots+a_r=p$, $r\ge2$, and $i_1>i_2>\dots>i_r\ge0$, then we have
$$F(n;p)\equiv
\frac {p!} {a_1!\,a_2!\cdots a_r!}
\pmod{p^2}.$$
\item[(iii)]
In the cases not covered by items {\em(i)} and {\em(ii),} 
the Fu\ss--Catalan number $F(n;p)$ is divisible by~$p^2$.
\end{enumerate}
\end{corollary}

\begin{proof}
By means of Equation~\eqref{eq:Phipot} and 
Proposition~\ref{conj:H},
we convert the right-hand side of \eqref{eq:Loes1FCata} into a linear
combination of series $H_{a_1,a_2,\dots,a_r}\!\left(z^{1/(p-1)}\right)$,
with all $a_i$'s relatively prime to~$p$. The result is
\begin{multline*} 
\sum _{n\ge1} ^{}\frac {1} {n}\binom {pn}{n-1}\,z^n
= 
z^{-1/(p-1)}H_{1}\!\left(z^{1/(p-1)}\right)-1\\
+z^{-1/(p-1)} \sum _{r=1} ^{p}
\underset{a_1+\dots+a_r=p\text{ and }r\ge2}{\sum _{a_1,\dots,a_r\ge1} ^{}}
\frac {p!}
{a_1!\,a_2!\cdots a_r!}H_{a_1,a_2,\dots,a_r}\!\left(z^{1/(p-1)}\right)
\quad \quad 
\text {modulo }p^2.
\end{multline*}
Coefficient extraction then leads to our claim.
\end{proof}
}

\section{The number of blossom trees modulo $p$-powers}
\label{sec:Schaef}

The combinatorial objects which we treat in this section are so-called
``blossom trees." These are a particular kind of trees which are
of great significance in the combinatorial understanding of the
enumeration of maps; see e.g.\ the survey \cite{BousAS}. 
The particular blossom trees that we are interested in are
the ones in \cite[Sec.~3]{SchaAA}. Since the precise definition
is slightly technical and not needed here, 
we omit it, and instead
refer the reader to \cite{SchaAA}.

For an odd positive integer $k$, let $B(n;k)$ be the number of
blossom trees constructed from $k$-ary trees with $n$ white nodes by
adding a black node with $k-1$ buds on each inner edge.
Let $B_k(z):=\frac {k+1} {2}+\sum _{n\ge1} ^{}B(n;k)\,z^n$ be the
corresponding generating function. Schaeffer \cite[Cor.~2]{SchaAA}
proved that 
\begin{equation} \label{eq:Snp}
B(n;k)=\frac {k+1} {n((k-1)n+2)}\binom {kn}{n-1}.
\end{equation}

As we show in the Appendix, the generating function
$B_k(z)$ satisfies the polynomial equation
\begin{multline} \label{eq:Sdiff}
z^2B^k_k(z) + 
\sum_{s=0}^{(k+1)/2}
(-1)^s\frac {k+1} {(k-s+1)(k-s)}\binom {k-s+1}s
k^{k-2s+1}\left(\frac {k-1} {2}\right)^{s}
zB_k^s(z) \\
-(-1)^k\left(\frac {k-1} {2}\right)^{k-1}B_k(z) 
+(-1)^k\frac {k+1} {2}\left(\frac {k-1} {2}\right)^{k-1}
= 0.
\end{multline}
It is not difficult to see that \eqref{eq:Sdiff} determines $B_k(z)$
uniquely as a formal power series over~$\Z$, and as well over any
power of a prime number~$p$ as long as \hbox{$k\not\equiv1$}~(mod~$p$).

\medskip
From now on, let $p$ be a fixed odd prime number.

\begin{theorem} \label{thm:Schaeff}
Let $\Phi(z)=\sum _{n\ge0} ^{}z^{p^n},$ and let $\al$ be a
non-negative integer.
Then the generating function $B_p(z),$ when reduced modulo $p^{p^\al},$ 
can be expressed as a polynomial in $\Phi\left(z^{1/(p-1)}\right)$ 
of degree at most
$p^{\al+1}-1,$ with coefficients that are Laurent polynomials in $z^{1/(p-1)}$
over the integers.
\end{theorem}

\begin{proof} 
Again, we apply the method from Section~\ref{sec:method}. 
To begin with, we need a ``base solution" 
\begin{equation} \label{eq:F1} 
F_1(z)=\sum _{i=0} ^{p^{\al+1}-1}a_{i,1}(z)\Phi^{i}\!\left(z^{1/(p-1)}\right)
\end{equation}
to \eqref{eq:Sdiff} modulo~$p$. At this point, it is useful to observe
that, using Fermat's little theorem, the functional equation 
\eqref{eq:Sdiff}, when taken modulo~$p$, reduces to
\begin{equation} \label{eq:Sdiff2}
z^2B^p_p(z) - 2^{(p+1)/2} zB_p^{(p+1)/2}(z) 
+B_p(z) 
-2^{-1}
= 0
\quad \quad \text{modulo }p,
\end{equation}
where here, and in the following, $2^{-1}$ denotes the inverse of~$2$
modulo~$p$. 
We claim that the following choices in \eqref{eq:F1} 
solve the above congruence:
\begin{align*}
a_{0,1}(z)&=2^{-1}z^{-2/(p-1)}s^2_\al(z^{1/(p-1)})
\quad \text {modulo }p,\\[2mm]
a_{p^{\al},1}(z)&=z^{-2/(p-1)}s_\al(z^{1/(p-1)})
\quad \text {modulo }p,\\
a_{2\cdot p^{\al},1}(z)&=2^{-1}z^{-2/(p-1)}
\quad \text {modulo }p,
\end{align*}
where $s_\al(z)=\sum _{k=0} ^{\al-1}z^{p^k}$,
with all other $a_{i,1}(z)$ vanishing.
In order to verify this claim, we first observe that, with the
above choices, $F_1(z)$ can be expressed as a square, namely
$$
F_1(z)=2^{-1}z^{-2/(p-1)}\left(
\Phi^{p^\al}\!\!\left(z^{1/(p-1)}\right)+s_\al\!\left(z^{1/(p-1)}\right)
\right)^2.
$$
We substitute this in
\eqref{eq:Sdiff2}. For the left-hand side, we obtain
\begin{multline*} 
2^{-1}z^{-2/(p-1)}\left(
\Phi^{p^\al}\!\big(z^{1/(p-1)}\big)+s_\al\big(z^{1/(p-1)}\big)
\right)^{2p} \\
- z^{-2/(p-1)}
\left(
\Phi^{p^\al}\!\!\left(z^{1/(p-1)}\right)+s_\al\big(z^{1/(p-1)}\big)
\right)^{p+1} \\
+2^{-1}z^{-2/(p-1)}\left(
\Phi^{p^\al}\!\big(z^{1/(p-1)}\big)+s_\al\big(z^{1/(p-1)}\big)
\right)^2
-2^{-1}
\quad \text {modulo }p.
\end{multline*}
Using the relation \eqref{eq:PhiRel2} with $z$ replaced by $z^{1/(p-1)}$
and reducing the obtained expression modulo~$p$,
we arrive at
\begin{align*} 
2^{-1}&z^{-2/(p-1)}\left(
\Phi^{p^{\al+1}}\!\big(z^{1/(p-1)}\big)
+s^p_\al\big(z^{1/(p-1)}\big)
\right)^{2} \\
&\kern1cm
- z^{-2/(p-1)}
\left(
\Phi^{p^{\al+1}}\!\!\left(z^{1/(p-1)}\right)+s^p_\al\big(z^{1/(p-1)}\big)
\right)
\left(
\Phi^{p^\al}\!\!\left(z^{1/(p-1)}\right)+s_\al\big(z^{1/(p-1)}\big)
\right) \\
&\kern1cm
+2^{-1}z^{-2/(p-1)}\left(
\Phi^{p^\al}\!\big(z^{1/(p-1)}\big)+s_\al\big(z^{1/(p-1)}\big)
\right)^2
-2^{-1}\\
&=2^{-1}z^{-2/(p-1)}\left(
\Phi^{p^{\al}}\!\big(z^{1/(p-1)}\big)
-z^{p^\al/(p-1)}
+s_{\al+1}\big(z^{1/(p-1)}\big)
-z^{1/(p-1)}
\right)^{2} \\
&\kern1cm
- z^{-2/(p-1)}
\left(
\Phi^{p^{\al}}\!\big(z^{1/(p-1)}\big)
-z^{p^\al/(p-1)}
+s_{\al+1}\big(z^{1/(p-1)}\big)
-z^{1/(p-1)}
\right)\\
&\kern4cm
\times
\left(
\Phi^{p^\al}\!\!\left(z^{1/(p-1)}\right)+s_\al\big(z^{1/(p-1)}\big)
\right) \\
&\kern1cm
+2^{-1}z^{-2/(p-1)}\left(
\Phi^{p^\al}\!\big(z^{1/(p-1)}\big)+s_\al\big(z^{1/(p-1)}\big)
\right)^2
-2^{-1}\\
&=2^{-1}z^{-2/(p-1)}\left(
\Phi^{p^{\al}}\!\big(z^{1/(p-1)}\big)
+s_{\al}\big(z^{1/(p-1)}\big)
-z^{1/(p-1)}
\right)^{2} \\
&\kern1cm
- z^{-2/(p-1)}
\left(
\Phi^{p^{\al}}\!\big(z^{1/(p-1)}\big)
+s_{\al}\big(z^{1/(p-1)}\big)
-z^{1/(p-1)}
\right)\\
&\kern4cm
\times
\left(
\Phi^{p^\al}\!\!\left(z^{1/(p-1)}\right)+s_\al\big(z^{1/(p-1)}\big)
\right) \\
&\kern1cm
+2^{-1}z^{-2/(p-1)}\left(
\Phi^{p^\al}\!\big(z^{1/(p-1)}\big)+s_\al\big(z^{1/(p-1)}\big)
\right)^2
-2^{-1}
\quad \text {modulo }p.
\end{align*}
By collecting terms, this expression may be simplified to zero.

\medskip
After we have completed the ``base  step," we now proceed with the
iterative steps described in Section~\ref{sec:method}. 
Our Ansatz
here (replacing the corresponding one in
\eqref{eq:Ansatz2}--\eqref{eq:Ansatz2b}) is
\begin{equation} \label{eq:AnsatzAp}
B_p(z)
=\sum _{i=0} ^{p^{(\al+1)h}-1}a_{i,\be+1}(z)
\Phi^i\left(z^{1/(p-1)}\right)\quad 
\text{modulo }p^{\be+1},
\end{equation}
with
\begin{equation} \label{eq:AnsatzApa}
a_{i,\be+1}(z):=a_{i,\be}(z)+p^{\be}b_{i,\be+1}(z),\quad 
i=0,1,\dots,p^{(\al+1)h}-1,
\end{equation}
where the
coefficients $a_{i,\be}(z)$ are supposed to provide a solution
$$F_\be(z)=\sum _{i=0}
^{p^{(\al+1)h}-1}a_{i,\be}(z)\Phi^i\left(z^{1/(p-1)}\right)$$ 
to
\eqref{eq:Sdiff} modulo~$p^\be$. This Ansatz, substituted in
\eqref{eq:Sdiff}, produces the congruence
\begin{multline} \label{eq:ItGlS}
2(-1)^{(p+1)/2}\left(\frac {p-1} {2}\right)^{(p-1)/2}
\frac {p+1} {2}\\
\kern3cm
\times
p^\be z
\left(\sum _{i=0} ^{p^{\al+1}-1}
b_{i,\be+1}(z)\Phi^{i}\!\!\left(z^{1/(p-1)}\right)\right)
\left(\sum _{j=0} ^{p^{\al+1}-1}
a_{j,\be}(z)\Phi^{j}\!\!\left(z^{1/(p-1)}\right)\right)^{(p-1)/2}\\
+\left(\frac {p-1} {2}\right)^{p-1}
p^\be\left(\sum _{i=0} ^{p^{\al+1}-1}
b_{i,\be+1}(z)\Phi^{i}\!\!\left(z^{1/(p-1)}\right)\right)
\kern5.5cm\\
+z^2F_\be^p(z) + 
\sum_{s=0}^{(p+1)/2}
(-1)^s\frac {p+1} {(p-s+1)(p-s)}\binom {p-s+1}s
p^{p-2s+1}\left(\frac {p-1} {2}\right)^{s}
zF_\be^s(z) \\
+\left(\frac {p-1} {2}\right)^{p-1}F_\be(z) 
-\frac {p+1} {2}\left(\frac {p-1} {2}\right)^{p-1}
=0
\quad \quad 
\text {modulo }p^{\be+1}.
\end{multline}
Since the sum has the prefactor $p^\be$, we may reduce the
$a_{i,\be}(z)$ modulo $p$. By construction, we have
\begin{align*}
a_{0,\be}(z)&=a_{0,1}(z)=2^{-1}z^{-2/(p-1)}s^2_\al(z^{1/(p-1)})
\quad \quad \text{modulo }p,\\
a_{p^\al,\be}(z)&=a_{p^\al,1}(z)=z^{-2/(p-1)}s_\al(z^{1/(p-1)})
\quad \quad \text{modulo }p,\\
a_{2\cdot p^\al,\be}(z)&=a_{2\cdot p^\al,1}(z)=2^{-1}z^{-2/(p-1)}
\quad \quad \text{modulo }p,
\end{align*}
and $a_{i,\be}(z)=0$ modulo $p$ for all other $i$'s. 
Equivalently, we have
$$
F_\be(z)=
2^{-1}z^{-2/(p-1)}
\left(\Phi^{p^\al}\!\!\left(z^{1/(p-1)}\right)
+s_\al\!\left(z^{1/(p-1)}\right)\right)^2
\quad \quad \text{modulo }p.
$$
If we substitute this in \eqref{eq:ItGlS} and divide both sides of
the congruence by $p^\be$ (which is possible by our assumption
on $F_\be(z)$), we obtain
\begin{multline*} 
-\left(\sum _{i=0} ^{p^{\al+1}-1}
b_{i,\be+1}(z)\Phi^{i}\!\!\left(z^{1/(p-1)}\right)\right)
\left(\Phi^{p^\al}\!\!\left(z^{1/(p-1)}\right)
+s_\al\!\left(z^{1/(p-1)}\right)\right)^{p-1}\\
+
\left(\sum _{i=0} ^{p^{\al+1}-1}
b_{i,\be+1}(z)\Phi^{i}\!\!\left(z^{1/(p-1)}\right)\right)
+G_\be(z)\\
=
-\left(\sum _{i=0} ^{p^{\al+1}-1}
b_{i,\be+1}(z)\Phi^{i}\!\!\left(z^{1/(p-1)}\right)\right)
\left(\sum_{j=0} ^{p-1}(-1)^j
\Phi^{j p^\al}\!\!\left(z^{1/(p-1)}\right)
s_\al^{p-j-1}\!\left(z^{1/(p-1)}\right)
\right)\\
+
\left(\sum _{i=0} ^{p^{\al+1}-1}
b_{i,\be+1}(z)\Phi^{i}\!\!\left(z^{1/(p-1)}\right)\right)
+G_\be(z)=0
\quad \quad 
\text {modulo }p,
\end{multline*}
where $G_\be(z)$ is some polynomial in $\Phi\left(z^{1/(p-1)}\right)$
whose coefficients are Laurent polynomials in $z^{1/(p-1)}$ with
integer coefficients.
Upon using
\eqref{eq:PhiRel2} with $z$ replaced by $z^{1/(p-1)}$ 
to reduce high powers of $\Phi\big(z^{1/(p-1)}\big)$, we 
convert this congruence into
\begin{multline*} 
-\sum_{j=0} ^{p-1}
\sum _{i=jp^\al} ^{p^{\al+1}-1}
\Phi^{i}\!\!\left(z^{1/(p-1)}\right)
(-1)^jb_{i-jp^\al,\be+1}(z)
s_\al^{p-j-1}\!\left(z^{1/(p-1)}\right)
\\
-\sum_{j=0} ^{p-1}
\sum _{i=p^{\al}} ^{(j+1)p^{\al}-1}
\Phi^{i}\!\!\left(z^{1/(p-1)}\right)
(-1)^jb_{i+(p-j-1)p^\al,\be+1}(z)
s_\al^{p-j-1}\!\left(z^{1/(p-1)}\right)
\\
+
\sum_{j=0} ^{p-1}
\sum _{i=0} ^{jp^\al-1}
\Phi^{i}\!\!\left(z^{1/(p-1)}\right)
(-1)^jz^{p^\al/(p-1)}b_{i+(p-j)p^\al,\be+1}(z)
s_\al^{p-j-1}\!\left(z^{1/(p-1)}\right)
\\
+
\left(\sum _{i=0} ^{p^{\al+1}-1}
b_{i,\be+1}(z)\Phi^{i}\!\!\left(z^{1/(p-1)}\right)\right)
+G_\be(z)
=0
\quad \quad 
\text {modulo }p,
\end{multline*}

Comparison of powers of $\Phi(z)$ then yields a system of congruences
of the form
\begin{equation} \label{eq:AbcS} 
M\cdot b=c\quad \text{modulo }p,
\end{equation}
where $b$ is the column vector of unknowns
$(b_{i,\be+1}(z))_{i=0,1,\dots,p^{\al+1}-1}$, $c$ is a (known)
column vector of Laurent polynomials in $z$, and $M$ is the matrix
{\small
$$
\begin{pmatrix} 
D(1-s^{p-1}_\al)&D(z^{A})&D(-z^{A}s_\al)&\dots&
D(z^{A}s^{p-3}_\al)&D(-z^{A}s^{p-2}_\al)\\
D(s^{p-2}_\al)&D(-s^{p-1}_\al)&D(s_\al+z^{A})&\dots&
D(-s^{p-3}_\al-z^{A}s^{p-4}_\al)&D(s^{p-2}_\al+z^{A}s^{p-3}_\al)\\
D(-s^{p-3}_\al)&D(s^{p-2}_\al)&D(-s^{p-1}_\al)&\dots&
D(s^{p-4}_\al+z^{A}s^{p-5}_\al)&D(-s^{p-3}_\al-z^{A}s^{p-4}_\al)\\
\vdots&\vdots&\vdots&\ddots\\
D(s_\al)&D(-s^2_\al)&D(s^3_\al)&\dots&D(-s^{p-1}_\al)&D(s_\al+z^A)\\
D(-1)&D(s_\al)&D(-s^2_\al)&\dots&D(s^{p-2}_\al)&D(-s^{p-1}_\al)
\end{pmatrix},
$$}%
where $s_\al$ is short for $s_\al\big(z^{1/(p-1)}\big)$,
$A$ is short for $p^\al/(p-1)$, and
$D(x)$ denotes the $p^\al\times p^\al$ diagonal matrix whose
diagonal entries equal $x$. More precisely, the $(r,t)$-block of the
matrix, $0\le r,t\le p-1$, is given by
$$
\begin{cases} 
D(1-s^{p-1}_\al),&\text{if }r=t=0,\\
D((-1)^{t-1}z^As^{t-1}_\al),&\text{if $r=0$ and }t>0,\\
D((-1)^{r-t+1}s^{p-1-r+t}_\al),&\text{if }r\ge t\text{ but not $r=t=0$},\\
D((-1)^{t-r-1}s^{t-r-1}_\al(s_\al+z^A)),&\text{if }0<r<t.
\end{cases}
$$

We claim that
$$
\det (M)=z^{p^\al}\quad \quad \text{modulo }p.
$$
In order to see this, 
we add $1-s^{p-1}_\al(z)$ times row
$i+(p-1)\cdot p^\al$ to row $i$, $i=0,1,\dots,p^{\al}-1$,
and we add $s_\al\big(z^{1/(p-1)}\big)$ times row
$i+p^\al$ to row $i$, $i=p^\al,p^\al+1,\dots,p^{\al+1}-1$.
This leads to the matrix
$$
\begin{pmatrix} 
0&D(s_\al+z^{A}-s^p_\al)&*&\dots&
*&*\\
0&0&D(s_\al+z^{A}-s^{p}_\al)&\dots&
0&0\\
0&0&0&\dots&
0&0\\
\vdots&\vdots&\vdots&\ddots\\
0&0&0&\dots&0&D(s_\al+z^A-s^p_\al)\\
D(-1)&D(s_\al)&D(-s^2_\al)&\dots&D(s^{p-2}_\al)&D(-s^{p-1}_\al)
\end{pmatrix}.
$$
If, in addition, we move rows $(p-1)\cdot p^\al,(p-1)\cdot p^\al+1,
\dots,p^{\al+1}-1$ to the top of the matrix, then
it attains an upper triangular form, from which
one infers that
\begin{align*}
\det(M)&=(-1)^{(p-1)\cdot p^\al\cdot p^\al}
\det\!{}^{p-1} \big(D(s_\al\big(z^{1/(p-1)}\big)
+z^{p^\al/(p-1)}-s^p_\al\big(z^{1/(p-1)})\big)\big)
\det\big(D(1)\big)\\
&=\big(s_\al\big(z^{1/(p-1)}\big)
+z^{p^\al/(p-1)}-s^p_\al\big(z^{1/(p-1)}\big)\big)^{(p-1)\cdot p^\al}\\
&=z^{p^\al}\quad \quad \text{modulo }p,
\end{align*}
as was claimed. As a consequence,
the system \eqref{eq:AbcS} is
(uniquely) solvable. Thus, we have proved that, for an arbitrary non-negative
integer $\al$, the algorithm of
Section~\ref{sec:method} will produce a solution $F_{{p^{\al}}}(z)$ 
to \eqref{eq:Sdiff} modulo $p^{p^\al}$ which is a
polynomial in $\Phi\big(z^{1/(p-1)}\big)$ 
with coefficients that are Laurent polynomials in $z^{1/(p-1)}$.
\end{proof}

A direct corollary of the $\al=0$ case of the preceding proof is
the following.

\begin{corollary} \label{cor:B7}
The blossom tree numbers $B(n;p)=\frac {p+1} {n((p-1)n+2)}\binom {pn}{n-1}$
obey the following congruences modulo $p${\em:}

\begin{enumerate}
\item[(i)]
If $n=\frac {1} {p-1}(p^{i_1}+p^{i_2}-2)$ 
with $i_1-1>i_2\ge0$,
then $B(n;p)\equiv1$~{\em(mod~$p$)}.

\item[(ii)]
If $n=\frac {1} {p-1}\left(2p^i-2\right)$ with $i\ge1$,
then $B(n;p)\equiv\frac {p+1} {2}$~{\em(mod~$p$)}.

\item[(iii)]
In the cases not covered by items {\em(i)} and {\em(ii),} 
the number $B(n;p)$ is divisible by~$p$.
\end{enumerate}
\end{corollary}

Obviously, one can do better by using an implementation of the
algorithm contained in the proof of Theorem~\ref{thm:Schaeff}.
As an illustration, we display
below the result obtained for the modulus $p^2$. Again, this result
was first guessed from the automatically obtained results for
$p=3,5,7$, but, once found, it is easily verified directly by substitution
in \eqref{eq:Sdiff}.

{
\allowdisplaybreaks
\begin{theorem} \label{thm:S27}
Let $\Phi(z)=\sum _{n\ge0} ^{}z^{p^n}$.
Then we have
\begin{multline} 
\label{eq:Loes1S}
\frac {p+1} {2} + \sum _{n\ge1} ^{}B(n;p)\,z^n
=
(p+1)z^{-1/(p-1)}\Phi\!\left(z^{1/(p-1)}\right)
-\frac {p+1} {2}z^{-2/(p-1)}\Phi^2\!\left(z^{1/(p-1)}\right)\\
+z^{-2/(p-1)}\Phi^{p+1}\!\left(z^{1/(p-1)}\right)
\quad \quad 
\text {\em modulo }p^2.
\end{multline}
\end{theorem}
}

Explicitly, this means the following.

\begin{corollary} \label{cor:B49}
The blossom tree numbers $B(n;p)=\frac {p+1} {n((p-1)n+2)}\binom {pn}{n-1}$
obey the following congruences modulo $p^2${\em:}

\begin{enumerate}
\item[(i)]
If $n=\frac {1} {p-1}\left((p+1)p^i-2\right)$ with $i\ge0$,
then $B(n;p)\equiv1$~{\em(mod~$p^2$)}.

\item[(ii)]
If $n=\frac {1} {p-1}\left(2p^i-2\right)$ with $i\ge1$,
then $B(n;p)\equiv\frac {p+1} {2}$~{\em(mod~$p^2$)}.

\item[(iii)]
If $n=\frac {1} {p-1}(p^{i_1}+p^{i_2}-2)$ 
with $i_1-1>i_2\ge0$,
then $B(n;p)\equiv p+1$~{\em(mod~$p^2$)}.

\item[(iv)]
If 
$$n=\frac {1} {p-1}\left(a_1p^{i_1}+a_2p^{i_2}+\dots+
a_rp^{i_r}-2\right)$$
with all $a_i$'s relatively prime to~$p$,
$a_1+a_2+\dots+a_r=p+1$, $r\ge2$, and $i_1>i_2>\dots>i_r\ge0$, then we have
$$B(n;p)\equiv
\frac {(p+1)!} {a_1!\,a_2!\cdots a_r!}
\pmod{p^2}.$$
\item[(v)]
In the cases not covered by items {\em(i)}--{\em(iv),} 
the number $B(n;p)$ is divisible by~$p^2$.
\end{enumerate}
\end{corollary}

\begin{proof}
By means of Equation~\eqref{eq:Phipot} and 
Proposition~\ref{conj:H},
we convert the right-hand side of \eqref{eq:Loes1S} into a linear
combination of series
$H_{a_1,a_2,\dots,a_r}\!\left(z^{1/(p-1)}\right)$ 
with all $a_i$'s
relatively prime to~$p$. The result is
\begin{multline*}
\sum _{n\ge0} ^{}B(n;p)\,z^n
=
 \frac {p+1} {2} z^{-2/(p-1)} H_{2}\!\left(z^{1/(p-1)}\right) 
+ (p+1) z^{-2/(p-1)} H_{1, 1}\!\left(z^{1/(p-1)}\right) \\
-p z^{-2/(p-1)} H_{p+1}\!\left(z^{1/(p-1)}\right) \\
+z^{-2/(p-1)} \sum _{r=1} ^{p+1}
\underset{r\ge2\text{ and }a_i\ne p\text{ for all }i}{
\underset{a_1+\dots+a_r=p+1}{\sum _{a_1,\dots,a_r\ge1} ^{}}}
\frac {(p+1)!}
{a_1!\,a_2!\cdots a_r!}H_{a_1,a_2,\dots,a_r}\!\left(z^{1/(p-1)}\right)
\quad \quad 
\text{modulo }p^2.
\end{multline*}
Coefficient extraction then yields the claimed congruences.
\end{proof}

\section*{Appendix: Proof of the functional equation \eqref{eq:Sdiff}}

\setcounter{equation}{0}%
\global\def\theequation{\mbox{A.\arabic{equation}}}

Let $k$ be a positive integer and let
$$
T_k(z)=\sum_{n\ge1}\frac {1} {n}\binom {kn}{n-1}z^n
$$
be the generating function for the general Fu\ss--Catalan numbers.
We claim that the blossom tree generating function 
$B_k(z)$ can be expressed
in terms of $T_k(z)$ as follows.

\begin{lemma} \label{lem:S-T}
For all positive integers $k$, we have
\begin{equation} \label{eq:S-T}
B_k(z)=(1+T_k(z))\left(k-\frac {k-1} {2}(1+T_k(z))\right).
\end{equation}
\end{lemma}

\begin{proof}
We verify that the coefficient of $z^n$ is the same on both sides of
\eqref{eq:S-T}, for $n=0,1,\dots$ Clearly, the constant coefficient equals
$(k+1)/2$ on both sides.

Now let $n\ge1$.
It is easily seen by Lagrange inversion (see
\cite[Theorem~5.4.2]{StanBI}), that the series $T_k(z)$ satisfies
the equation 
\begin{equation} \label{eq:Tk} 
T_k(z)=z(1+T_k(z))^k.
\end{equation}
In other words, it is the
compositional inverse of $z/(1+z)^k$. By applying 
Lagrange inversion again, we infer
\begin{align*}
\coef{z^n} (1+T_k(z))\left(k-\frac {k-1} {2}(1+T_k(z))\right)
&
=\frac {1} {n}\coef{z^{-1}}\left(1-(k-1)z\right)
\frac {(1+z)^{kn}} {z^n}\\
&\kern-1cm
=\frac {1} {n}\left(\coef{z^{n-1}}(1+z)^{kn}
-(k-1)\coef{z^{n-2}}(1+z)^{kn}\right)\\
&\kern-1cm
=\frac {1} {n}\left(\binom{kn}{n-1}
-(k-1)\binom {kn}{n-2}\right)\\
&\kern-1cm
=\frac {k+1} {n((k-1)n+2)}\binom{kn}{n-1}=B(n;k).
\qedhere\end{align*}
\end{proof}

\begin{proof}[Proof of \eqref{eq:Sdiff}]
In order to establish \eqref{eq:Sdiff},
we substitute the right-hand side of \eqref{eq:S-T} for $B_k(z)$
in \eqref{eq:Sdiff}. After applying the binomial theorem, we
obtain 
\begin{multline*}
z^2\sum_{j=0}^k(-1)^j\binom kjk^{k-j}
\left(\frac {k-1} {2}\right)^jX^{k+j}(z)\\
 + 
\sum_{s=0}^{(k+1)/2}\sum_{j=0}^s
(-1)^{s+j}\frac {k+1} {(k-s+1)(k-s)}
\kern5cm\\
\kern5cm
\cdot
\binom {k-s+1}s\binom sj
k^{k-s-j+1}\left(\frac {k-1} {2}\right)^{s+j}
zX^{s+j}(z) \\
-(-1)^k\left(\frac {k-1} {2}\right)^{k-1}
X(z)\left(k-\frac {k-1} {2}X(z)\right) 
+(-1)^k\frac {k+1} {2}\left(\frac {k-1} {2}\right)^{k-1}
\end{multline*}
on the left-hand side, where $X(z)=1+T_k(z)$. In order to simplify, we use
the relation \eqref{eq:Tk} in the first term, 
we replace $s$ by $s-j$, and then write the sum over $j$
in standard hypergeometric notation
$${}_p F_q\!\left[\begin{matrix} a_1,\dots,a_p\\ b_1,\dots,b_q\end{matrix}; 
z\right]=\sum _{m=0} ^{\infty}\frac {\po{a_1}{m}\cdots\po{a_p}{m}}
{m!\,\po{b_1}{m}\cdots\po{b_q}{m}} z^m\ ,$$
where the Pochhammer symbol 
$(\alpha)_m$ is defined by $(\alpha)_m:=\alpha(\alpha+1)\cdots(\alpha+m-1)$,
$m\ge1$, and $(\alpha)_0:=1$. In this manner, we arrive at
\begin{multline*}
z\sum_{j=0}^k(-1)^j\binom kjk^{k-j}
\left(\frac {k-1} {2}\right)^j(X(z)-1)X^{j}(z)\\
+
2(-1)^{k+1}
\left(\frac {k-1} {2}\right)^{k}
zX^{k+1}(z) 
 + 
(-1)^{k}k(k+1)
\left(\frac {k-1} {2}\right)^{k-1}
zX^{k}(z) \\
+\sum_{s=0}^{k-1}
(-1)^{s}\frac {(k+1)\,(k-s-1)!} {(k-2s+1)!\,s!}
k^{k-s+1}\left(\frac {k-1} {2}\right)^{s}\kern3cm\\
\cdot
zX^{s}(z) 
\,  {} _{3} F _{2} \!\left [ \begin{matrix}
k-s,-\frac {s} {2},-\frac {s} {2}+1
\\ 
\frac {k} {2}-s+\frac {1} {2},\frac {k} {2}-s+1
\end{matrix} 
; 1\right ]  
\\
-(-1)^k\left(\frac {k-1} {2}\right)^{k-1}
X(z)\left(k-\frac {k-1} {2}X(z)\right) 
+(-1)^k\frac {k+1} {2}\left(\frac {k-1} {2}\right)^{k-1}.
\end{multline*}
The $_3F_2$-series can be evaluated by means of the
Pfaff--Saalsch\"utz summation formula (see \cite[(2.3.1.3),
Appendix~(III.2)]{SlatAC})
$$
{} _{3} F _{2} \!\left [ \begin{matrix} { a, b, -n}\\ { c, 1 + a + b - c -
   n}\end{matrix} ; {\displaystyle 1}\right ]  =
  {\frac{({ \textstyle c-a}) _{n}  \,({ \textstyle c-b}) _{n} }
    {({ \textstyle c}) _{n} \, ({ \textstyle c-a-b}) _{n} }},
$$
provided $n$ is a non-negative integer. Thus, we obtain
\begin{multline*}
\sum_{j=0}^k(-1)^j\binom kjk^{k-j}
\left(\frac {k-1} {2}\right)^jz(X(z)-1)X^{j}(z)\\
+
2(-1)^{k+1}
\left(\frac {k-1} {2}\right)^{k}
zX^{k+1}(z) 
 + 
(-1)^{k}k(k+1)
\left(\frac {k-1} {2}\right)^{k-1}
zX^{k}(z) \\
 + 
\sum_{s=0}^{k-1}
(-1)^{s}\frac {(k+s-1)} 
{2}
k^{k-s}
\left(\frac {k-1} {2}\right)^{s-1}
\binom {k+1}s
zX^{s}(z) 
\\
-(-1)^k\left(\frac {k-1} {2}\right)^{k-1}
X(z)\left(k-\frac {k-1} {2}X(z)\right) 
+(-1)^k\frac {k+1} {2}\left(\frac {k-1} {2}\right)^{k-1}\\
=
\sum_{j=0}^k(-1)^j\binom kjk^{k-j}
\left(\frac {k-1} {2}\right)^jz(X(z)-1)X^{j}(z)
\kern5cm\\
+
(-1)^{k+1}
\left(\frac {k-1} {2}\right)^{k}
zX^{k+1}(z) 
 + 
(-1)^{k}\frac {k+1} {2}
\left(\frac {k-1} {2}\right)^{k-1}
zX^{k}(z) \\
 + 
\sum_{s=0}^{k+1}
(-1)^{s}\frac {(k+s-1)} 
{2}
k^{k-s}
\left(\frac {k-1} {2}\right)^{s-1}
\binom {k+1}s
zX^{s}(z) 
\\
-(-1)^k\left(\frac {k-1} {2}\right)^{k-1}
X(z)\left(k-\frac {k-1} {2}X(z)\right) 
+(-1)^k\frac {k+1} {2}\left(\frac {k-1} {2}\right)^{k-1}.
\end{multline*}
In the last expression, the sums can in fact be extended to run
over {\it all\/} integers $j$ respectively $s$. A comparison of
powers of $X(z)$ then makes it obvious that the sum over $j$ cancels
with the sum over $s$. To the remaining terms $X^{k+1}(z)$ and
$X^k(z)$ we apply again the relation \eqref{eq:Tk}. This turns the
above expression into
\begin{multline*}
(-1)^{k+1}\left(\frac {k-1} {2}\right)^{k}
(X(z)-1)X(z) 
+(-1)^k
\frac {k+1} {2}
\left(\frac {k-1} {2}\right)^{k-1}
(X(z)-1) \\
-(-1)^k\left(\frac {k-1} {2}\right)^{k-1}
X(z)\left(k-\frac {k-1} {2}X(z)\right) 
+(-1)^k\frac {k+1} {2}\left(\frac {k-1} {2}\right)^{k-1}=0,
\end{multline*}
which finishes the proof of \eqref{eq:Sdiff}.
\end{proof}

\end{document}